\documentclass[preprint,review,11pt,authoryear]{elsarticle}

\usepackage{amsthm}
 
\newtheorem{proposition}{Proposition}

\newtheorem{corollary}{Corollary}

\newtheorem{theorem}{Theorem}

\usepackage{algorithm}
\usepackage{algpseudocode}
\usepackage{tikz}
\usepackage{caption}
\usepackage{subcaption}
\usepackage{todonotes}

\usepackage{booktabs,threeparttable,multirow}

\usepackage{amssymb,amsmath,array, verbatim}
\usepackage{lscape}
\usepackage{xr}
\usepackage{url}
\usepackage{color}

\biboptions{comma,round}
\setcitestyle{citesep={;}}

\journal{arXiv}
\usepackage[top=1in, bottom=1in, left=1in, right=1in]{geometry}

\begin{document}
\begin{frontmatter}

\title{Revisiting Continuous $p$-Hub Location Problems with the $\ell^1$ Metric}
\author[add1]{Yifan Wu}

\author[add1]{Joseph Geunes}
\author[add2,add1]{Xiaofeng Nie\corref{cor1}}

\address[add1]{Wm Michael Barnes '64 Department of Industrial and Systems Engineering, Texas A\&M University, College Station, TX 77843, USA}
\address[add2]{Department of Engineering Technology and Industrial Distribution, Texas A\&M University, College Station, TX 77843, USA}

\cortext[cor1]{Corresponding author.\ead{xiaofengnie@tamu.edu}}
\begin{abstract}

Motivated by emerging urban applications in commercial, public sector, and humanitarian logistics, we revisit continuous $p$-hub location problems in which several facilities must be located in a continuous space such that the expected minimum Manhattan travel distance from a random service provider to a random customer through exactly one hub facility is minimized. In this paper, we begin by deriving closed-form results for a one-dimensional case and two-dimensional cases with up to two hubs. Subsequently, a simulation-based approximation method is proposed for more complex two-dimensional scenarios with more than two hubs. Moreover, an extended problem with multiple service providers is analyzed to reflect real-life service settings. Finally, we apply our model and approximation method using publicly available data as a case study to optimize the deployment of public-access automated external defibrillators in Virginia Beach.

\end{abstract}
\begin{keyword}
Continuous hub location problems\sep Simulation optimization  \sep  Automated external defibrillators
\end{keyword}
\end{frontmatter}

\section{Introduction}

With the advancement of GPS technology for real-time position tracking, a variety of innovative location-based application problems have emerged. These challenging problems are not only limited to commercial and public sector logistics but also extend to humanitarian logistics. In the commercial sector, companies such as Walmart and Amazon have combined crowd-sourced delivery with flexible warehouse stations located throughout metropolitan areas in order to strike a balance between cost and responsiveness (see, e.g., \cite{fatehi2022crowdsourcing}). Another example in the public sector deals with the optimal placement of automated external defibrillators (AEDs) in public areas to save lives after cardiac arrests. Additionally, in the realm of humanitarian logistics, a key question arises regarding the optimal pre-positioning of emergency medical supplies such as bandages, tourniquets, and antibiotics within urban areas to enhance preparedness for catastrophic disasters.

These examples share several common characteristics. First, both service providers (specific drivers, bystanders familiar with AED usage, and registered or spontaneous volunteers with some basic emergency medical training) and those in need of service (customers, cardiac arrest victims, and casualties) are randomly distributed throughout urban areas. Second, potential facility location sites may be virtually anywhere in the region that satisfies relevant size, convenience, or accessibility constraints. Third, the service provider first needs to retrieve some physical good(s) from a facility before proceeding to serve the customer. Given the time-sensitive nature of many of these applications, the service provider will evaluate all possible travel paths from his/her position to the customer’s position via each facility and select the one with the shortest total distance.  Fourth, in urban settings, the $\ell^1$ metric, that is, the Manhattan metric, is suitable for measuring distance. 

In the above-mentioned urban logistics applications, several facilities must be located in a continuous space such that the expected travel distance is minimized. Since we can treat each facility as a hub, such applications motivate us to revisit continuous $p$-hub location problems.  Unlike traditional hub location problems (e.g., \citet{campbell1996hub}) that assume economic benefits for transportation between hubs, these urban logistics applications typically require only one hub for pickup of an item. This gives rise to the ``1-stop $p$-hub problem" as termed by \citet{sasaki1999selection}. Additionally, in these applications, origin-destination demand pairs (service providers and customers) are not restricted to pre-specified locations but are present throughout a continuous service region with uncertainty, rendering the problem distinct. Thus, the the problem class we consider falls into a special case of continuous $p$-hub location problems with the $\ell^1$ metric, but with random origin-destination demand pair locations.

Despite the increasing prevalence of urban logistics applications,  relatively little academic research has been dedicated to the related location problems that arise. Therefore, this paper aims to revisit continuous $p$-hub location problems under the $\ell^1$ metric, where several facilities are to be located in a continuous space such that the expected minimum Manhattan travel distance from a random service provider to a random customer through exactly one hub facility is minimized.

This paper first undertakes a comprehensive exploration of the problem in one and two dimensions. Starting from a one-dimensional setting, that is, a line, the objective function is derived through an integration method, based on properties that specify the conditions under which a specific hub would be utilized. As a result, the convexity of the one-dimensional problem is shown, and the corresponding closed-form solution for the optimal locations is obtained. In a two-dimensional setting, we extend the integration method into the plane for up to two hubs. Some geometrical observations determine the boundaries of the integration areas for the two-hub case. Therefore, we are able to obtain an analytical expression for the objective function. 

When more than two hubs are considered, deriving an objective function analytically is intractable because the required number of integration areas grows exponentially. Therefore, we propose a two-step approximation method based on simulation to tackle the general case.  In the first step, using Benders decomposition, we solve a large-scale $p$-median problem that corresponds to a discrete version of our problem to obtain a high-quality initial solution. Then, a simulation optimization method is applied to obtain an improved continuous solution in the plane. Computational experiments show the performance stability and superiority of our algorithm. We then consider an extended version of the problem that accounts for multiple service providers in order to more accurately reflect real-life service situations. We show that the extended problem can also be reformulated as a $p$-median problem using discretization, which enables finding a solution directly using our approximation method.

As a case study, we implement our model and approximation method on a problem involving the deployment of public-access AEDs using smartphone dispatch applications (apps). In recent years, various mobile apps, including PulsePoint in the United States \citep{pulsepoint}, Heartrunner in Sweden \citep{HeartRunner}, and Staying Alive in France \citep{StayingAlive}, have been developed to facilitate AED delivery prior to the arrival of Emergency Medical Services. Upon receiving notification of an out-of-hospital cardiac arrest (OHCA), these apps alert volunteer responders and identify the closest AED for them to retrieve and deliver to the victim. Numerous studies (e.g., \citet{berglund2018smartphone}, \citet{hatakeyama2018smartphone}, and \citet{smida2022pulsepoint}) have demonstrated the importance of these smartphone dispatch apps in reducing delivery times and enhancing survival rates. Although several studies (e.g., \citet{chan2016optimizing} and \citet{chan2018robust}) have considered the placement of AEDs in public areas using optimization-based methods, to our knowledge, none have considered their use in conjunction with smartphone dispatch apps that identify randomly dispersed service providers. With the improvement and popularization of these apps, we believe that dispatch apps will bring about significant changes in OHCA rescue operations in the future, rendering our case study necessary. Our experimental results show some critical insights into the benefits and tradeoffs associated with volunteer and AED numbers. Moreover, we perform a sensitivity analysis that demonstrates that our approximation method is robust, enhancing the practical applicability of our implementation.

The remainder of this paper is organized as follows. We first review  related literature in Section \ref{sec:Literature Review}. Section \ref{pd} describes the problem setting and provides a generic mathematical formulation of the problem. Sections \ref{onedimension} and \ref{twodimension} present closed-form results for a one-dimensional case and two-dimensional one- and two-hub cases, respectively. We propose our simulation-based approximation method to tackle two-dimensional cases with more than two hubs in Section \ref{sec:Simulation approximation}, and analyze the extended problem with multiple service providers in Section \ref{sec:multi_service}. Section \ref{sec:Case Study} presents the case study on AED deployment in Virginia Beach with realistic OHCA data. Section \ref{sec:Conclusion} concludes the paper and discusses future research directions.

\section{Literature Review}\label{sec:Literature Review}

The early history of continuous hub location problems dates back to the 1950s when \citet{miehle1958link} investigated how to interconnect a group of fixed points scattered in a plane by constructing a network of links with a given number of junctions (whose continuous locations will be decided), such that the sum of all link lengths using the Euclidean metric (that is, the $\ell^2$ metric) is minimal. Motivated by real-world examples from airline and express package delivery companies, \citet{o1986location} considers both continuous single- and two-hub location problems with the Euclidean metric where flows between city pairs are given. The continuous single-hub location problem case becomes a classical Weber problem. For the two-hub problem, a numerical solution procedure is provided to find two hub locations under a non-overlapping partition assumption.  \citet{o1991solution} formulate a minimax hub location problem where a single hub is located on the plane, such that the most expensive flow via the hub among all city pairs is minimized. \citet{o1992clustering} proposes a clustering approach for a planar $p$-hub location problem where the squared distance is utilized. \citet{aykin1992interacting} investigate continuous $p$-hub location problems on a plane and a sphere with the Euclidean and geodesic distances, respectively, and propose an effective algorithm that iteratively determines location and allocation decisions.

Instead of assuming a given set of origins and destinations, \citet{campbell1990freight} assumes that origins and destinations are uniformly distributed over a line segment, and studies the problem of locating freight consolidation terminals under three different routing schemes, i.e., nearest terminal, minimum distance, and minimum cost (where the nearest terminal approach requires choosing the closest terminal to the origin and closest to the destination, although these choices may not minimize distance traveled or total shipment cost). The average transportation cost under each scheme is derived analytically, and a closed-form solution for the optimal terminal locations is provided. Furthermore, the results obtained from the one-dimensional problems can be applied to the two-dimensional settings with the $\ell^1$  metric where terminals are assumed to be arranged in a rectangular lattice. \citet{campbell1990locating} investigates the terminal location problem under the Euclidean metric using an illustrative example with fifty discrete demand locations randomly scattered in a plane. 


The one-dimensional version of our problem is a special case of the problem class considered in \citet{campbell1990freight}. Interestingly, we arrive at the same closed-form optimal hub locations for the one-dimensional case, albeit through a different derivation process. To elaborate, Campbell derives the objective function analytically using the sizes of all intervals (i.e., between the left boundary and the first hub location, between hub locations, and between the last hub location and the right boundary) as decision variables. In contrast, we directly utilize the locations of hubs as our decision variables. Additionally, Campbell explores two-dimensional cases assuming that hub locations are arranged in a rectangular lattice. This assumption allows for the independent application of one-dimensional solutions to both horizontal and vertical cases. In comparison, we relax this assumption for the two-dimensional cases.

To the best of our knowledge, only the one-dimensional case of our problem has been addressed in the previous literature (i.e., \citet{campbell1990freight}). Moreover, no prior studies have explored the two-dimensional case without imposing additional strong assumptions that our model and solution methods do not require. Therefore, our contributions are twofold: 1) We propose a new analytical approach for deriving the objective function for this problem; and 2) we extend the analysis to the two-dimensional case, broadening the scope of existing research.

\section{Problem Description}
\label{pd}

We introduce some basic settings and provide a description of our problem. Our focus is on a continuous urban region over which a service provider and a customer are randomly distributed, and where the Manhattan distance metric is applicable. Rather than a direct route from the service provider's location to the customer's, the service provider is required to first visit a facility to gather necessary item(s) to provide service before proceeding to the customer's location. In cases where there are multiple facility options, the provider will select the facility that minimizes the total distance traveled (including both the distance from the provider to the facility and from the facility to the customer). We wish to locate $n$ facilities within the region to minimize the expected minimum travel distance from a random service provider location to a random customer location.

Let $\mathfrak{X}$ denote the location of a random customer and $\mathfrak{Y}$ denote the location of a random service provider. Let $\mathfrak{p}_1, \mathfrak{p}_2, \ldots, \mathfrak{p}_n$ be $n$ facility locations, which are to be determined. The corresponding generic optimization model is formulated as follows:
\begin{equation}
\label{obj}
\min_{\mathfrak{p}_1,\ldots,\mathfrak{p}_n}E\left[\min_{i \in[n]}\Big\{\|\mathfrak{X}-\mathfrak{p}_i\|_{1}+\|\mathfrak{Y}-\mathfrak{p}_i\|_{1}\Big\}\right],
\end{equation}
where $[n]= \{1,2,\ldots, n\}$ and $\| \cdot \|_{1}$ denotes the $\ell^1$ distance metric.  
In the following three sections, we will focus on one- and two-dimensional versions of the problem, respectively.

\section{One-Dimensional Problem} 
\label{onedimension}

For the one-dimensional case, without loss of generality, we focus on a line segment with length 1. Let $X$ be the location of a random customer and $Y$ be the location of a random service provider. We assume that $X$ and $Y$ are independent and both follow a uniform distribution on $[0,1]$. Let $0 \le p_1 < p_2 < \cdots < p_n \le 1$ be the locations of $n$ facilities, which are to be determined. The corresponding objective function becomes
\begin{equation*}
 \int_{0}^{1} \int_{0}^{1} \min_{i \in [n]}\Big\{|x-p_i|+|y-p_i|\Big\} \,dy \,dx.
\end{equation*}

In the following subsections, we first introduce some propositions that facilitate the characterization of this objective function. Then, we outline a procedure for deriving an algebraic expression for the objective function solely in terms of the facility locations. Finally, we prove the convexity of the objective function and derive a closed-form solution for the optimal locations.

\subsection{Properties}

If we define $d_{i}(x,y)$ as $|x-p_i|+|y-p_i|$ for $i = 1,2,\ldots,n$ (with a slight abuse of notation, we use $d_i$ for short), the objective function becomes  $$\int_{0}^{1} \int_{0}^{1} \min_{i \in [n]} d_{i} \,dy \,dx.$$
In order to derive an analytical expression for the objective function, for each given $x$, we need to divide the whole range of $y$ into non-overlapping intervals (except for some boundary points) such that, within each interval, a single location with $\min_{i \in [n]} d_{i}$ can be identified. We present several propositions and corollaries to characterize these intervals, the proofs of which can be found in Appendix \ref{Appendix:Proposition}.

\begin{proposition}
\label{p1}
For any $k \in [n-1]$, if $0 \le x \le p_k$ and $0 \le y \le 1$, we have $d_k \le d_{k+i}, \forall i \in [n-k]$.  
\end{proposition}

\begin{corollary}
\label{cor1}
    If $ 0 \le x \le p_1$ and $0 \le y \le 1$, we have $
    d_1  = \min_{i \in [n]} d_{i}. $ 
\end{corollary}

\begin{proposition}
\label{p2}
For any $k \in \{2,\ldots,n\}$, if $p_k \le x \le 1$ and $0 \le y \le 1$, we have $d_k \le d_{k-i},\forall i \in [k-1].$ 
\end{proposition}

\begin{corollary}
\label{cor2}
If $p_n \le x \le 1$ and $0 \le y \le 1,$ we have
$d_n = \min_{i \in [n]} d_i.$
\end{corollary}
The two preceding propositions and their corollaries imply that if the customer is to the left of (to the right of) location $p_1$ ($p_n$), then location 1 (location $n$) is the distance-minimizing facility for the service provider.

\begin{proposition}
\label{p3}
For any $k \in [n-1],$ if $p_{k} \le x \le p_{k+1}$ and $0 \le y \le p_{k}+p_{k+1}-x$, we have $d_k = \min_{i \in [n]} d_{i}$.
\end{proposition}

\begin{proposition}
\label{p4}
For any $k \in [n-1],$ if $p_{k} \le x \le p_{k+1}$ and $p_{k}+p_{k+1}-x \le y \le 1$, we have $d_{k+1} = \min_{i \in[n]} d_{i}$.
\end{proposition}
The two preceding propositions imply that if the customer is between facilities $k$ and $k+1$, then either facility $k$ or $k+1$ serves as the distance-minimizing facility for the service provider, while the location of the service provider determines which of these two choices minimizes the distance.




\subsection{Objective Function Derivation}
If we locate $n$ facilities, let $F_{n}$ be the corresponding objective function value for given $p_{1},p_{2},\ldots,p_{n}$ values, that is,
$$F_{n} = \int_{0}^{1} \int_{0}^{1} \min_{i \in[n]} d_{i} \,dy \,dx.$$
When $n=1$, we can calculate $F_1$ as follows:
{ \begin{align*}
F_1 &= \int_{0}^{1} \int_{0}^{1} d_{1} \,dy \,dx \\
    & =2 p_{1}^{2}-2 p_{1}+1.
\end{align*}}For $n\ge 2$, based on Corollaries \ref{cor1} and \ref{cor2} and Propositions \ref{p3} and \ref{p4}, we can write $F_n$ as
\begin{align*}
F_n = \int^{p_1}_0 \int^1_0 d_1 \,dy \,dx + \sum_{k=1}^{n-1} I_k \\
+ \int_{p_{n}}^1 \int^1_0 d_n \,dy \,dx,
\end{align*}
where for $k\in [n-1]$, $I_k$ is defined as
\begin{align*}
I_k =\int^{p_{k+1}}_{p_{k}} \int^{p_{k}+p_{k+1}-x}_0 d_{k} \,dy \,dx \\
+\int^{p_{k+1}}_{p_{k}} \int_{p_{k}+p_{k+1}-x}^1 d_{k+1} \,dy \,dx.
\end{align*}

For $2 \le j \le n$, we can calculate $\Delta F_j$, which is defined as $F_j-F_{j-1}$, as follows:
\begin{align*}
\Delta F_j = I_{j-1}+\int_{p_{j}}^1 \int^1_0 d_j \,dy\,dx \\
-\int_{p_{j-1}}^1 \int^1_0 d_{j-1} \,dy \,dx.  
\end{align*}
As shown in Appendix \ref{Appendix:F_j}, the above integral expression can be written as 
\begin{align}
\Delta F_j = \frac{1}{3} p_{j -1}^{3} -\frac{1}{3} p_{j}^{3} + p_{j-1}^{2} p_{j} - p_{j-1} p_{j}^{2} \nonumber\\-2 p_{j-1}^{2} +2 p_{j}^{2} + 2 p_{j-1}-2 p_{j}. \label{Delta}
\end{align}
Therefore, for $n \ge 2$, we have the following expression for $F_n$: 
\begin{align}
F_n=F_1 + \sum_{j=2}^n \Delta F_{j}. \label{F_n}
\end{align}

\subsection{Closed-Form Optimal Solution}

We first characterize the convexity of the objective function through the following theorem.

\begin{theorem}
The objective function for the one-dimensional problem is convex on the feasible region.
\end{theorem}

\begin{proof}{Proof.}
Let $\mathbf{H}_n$ denote the Hessian matrix of $F_n$. Based on Equation (\ref{Delta}), we observe that $\Delta F_j$ involves only decision variables $p_j$ and $p_{j-1}$. Moreover, based on Equation (\ref{F_n}) and the fact that $F_1$ involves only the decision variable $p_1$, $\mathbf{H}_n$ is a tridiagonal matrix (the expressions for the entries of $\mathbf{H}_n$ are provided in Appendix \ref{Appendix:H_n}). Since all second partial derivatives are continuous, $\mathbf{H}_n$ is symmetric. We consider the following three cases:

\noindent Case 1. For row 1, we check that  $\mathbf{H}_n(1,1)=2 p_{1}+2 p_{2}>0$ and $|\mathbf{H}_n(1,1)|-|\mathbf{H}_n(1,2)|=2p_1 + 2p_2 - (2p_2 - 2p_1) = 4p_1 \ge 0.$ 

\noindent Case 2. For row $2 \le i \le n-1$, we check that $\mathbf{H}_n(i,i) = 2p_{i+1} -2p_{i-1} > 0$ and $|\mathbf{H}_n(i,i)| - (|\mathbf{H}_n(i,i-1)| + |\mathbf{H}_n(i,i+1)|) = 2p_{i+1} - 2p_{i-1}  - (2p_i - 2p_{i-1} + 2p_{i+1} - 2p_i) = 0$. 

\noindent Case 3. For row $n$, we check that $\mathbf{H}_n(n,n)= 4 -2p_{n-1}- 2p_n > 0$ and $|\mathbf{H}_n(n,n)| - |\mathbf{H}_n(n,n-1)| =  4 -2p_{n-1} - 2p_n - (2p_n - 2p_{n-1}) = 4 -4p_n \ge 0$. 

\noindent Therefore, $\mathbf{H}_n$ is a real, symmetric, diagonally dominant matrix with non-negative diagonal entries.

Since real, symmetric matrices are a special case of Hermitian matrices, based on the Gershgorin Disc Theorem (see, for example, Theorem 6.1.10 in \citet{horn2012matrix}), a real symmetric diagonally dominant matrix with non-negative diagonal entries is positive semi-definite. Therefore, $\mathbf{H}_n$ is positive semi-definite. Thus, the objective function is convex. 
\end{proof}

Due to the convexity of the objective function, an optimal solution for $p_i$ can be obtained by solving $\nabla F_n = \mathbf{0}$. The closed-form optimal solution is as follows (calculation details are provided in Appendix \ref{Appendix:Closed One}): 
\begin{equation}
    \label{oned}
    p_i^*=\frac{(i-1)\sqrt 2+1}{(n-1)\sqrt 2+2}, \quad i \in [n].
\end{equation}
This result is consistent with the corresponding result from \citet{campbell1990freight}. Moreover, we can check that 
\begin{equation*}
p_i^*+p_{n-i+1}^* =\frac{(n-1)\sqrt2+2}{(n-1)\sqrt2+2}=1,  \quad i \in [n], 
\end{equation*}
that is, the optimal hub locations are symmetric with respect to the center of the $[0,1]$ segment.

Although Equation \eqref{oned} is the same solution that \cite{campbell1990freight} derived, our approach differs fundamentally, and, as we next show, is amenable to generalization to two-dimensional location problems without added restrictions that limit location choices.  

\section{Two-Dimensional One- and Two-Hub Problems}
\label{twodimension}

For the two-dimensional case, we focus on a square region with an edge length of 1. Let $\mathbf{X}=(X_{1},X_{2})$ be the location of a random customer and $\mathbf{Y}=(Y_{1},Y_{2})$ be the location of a random service provider in the region. We assume that random variables $X_{1}$, $X_{2}$, $Y_{1}$, and $Y_{2}$ are mutually independent and uniformly distributed over $[0,1]$. Let $\mathbf{x}=(x_1,x_2)$ and $\mathbf{y}=(y_1,y_2)$, and let $\mathbf{p}_1=(p_{11},p_{12})$, $\mathbf{p}_2=(p_{21},p_{22})$, $\ldots$, $\mathbf{p}_n=(p_{n1},p_{n2})$ be $n$ facility locations on the square, which are to be determined. The corresponding objective function becomes
\begin{equation*}
\int_{0}^{1} \int_{0}^{1} \int_{0}^{1} \int_{0}^{1}  \min_{i \in[n]}d_i \, dy_1 \, dy_2 \, dx_1 \, dx_2,
\end{equation*}
where $d_i = \|\mathbf{x} -\mathbf{p}_{i}\|_1 + \|\mathbf{y} -\mathbf{p}_{i}\|_1$.

In the following subsections, we first investigate the problem of locating one hub and then the problem with two hubs. A more general case that focuses on a non-square rectangular region is considered in Appendix \ref{Appendix:two hub rectangle}.

\subsection{One hub}


Suppose we wish to locate a single hub, that is, $\mathbf{p}_1=(p_{11},p_{12})$. The corresponding objective function $F_1$ reduces to
\begin{equation*}
2 \int_{0}^{1} \int_{0}^{1}   \big(|x_1 -p_{11}| + |x_2 - p_{12}| \big)  \, dx_1 \, dx_2, 
\end{equation*}
which can be expanded as 
\begin{align}
    F_1 &=  2 \Bigg[\int^{p_{12}}_{0}\int^{p_{11}}_0 (p_{11}-x_1+p_{12}-x_2) \, dx_1 \, dx_2 \nonumber \\ 
    &+\int^{1}_{p_{12}}\int^{p_{11}}_0 (p_{11}-x_1+x_2-p_{12}) \, dx_1 \, dx_2 \nonumber \\
    &  +\int^{p_{12}}_{0}\int^{1}_{p_{11}} (x_1-p_{11}+p_{12}-x_2) \, dx_1 \, dx_2  \nonumber \\
    &+\int^{1}_{p_{12}}\int^{1}_{p_{11}} (x_1-p_{11}+x_2-p_{12}) \, dx_1 \, dx_2 \Bigg] \nonumber \\
        &=  2 p_{11}^{2} + 2 p_{12}^{2} - 2 p_{11} -2 p_{12} +2. \label{F1Exp}
\end{align}

Based on Equation (\ref{F1Exp}), the Hessian matrix of $F_1$ is 
$\left[\begin{array}{cc}
4  & \quad 0  \\
0  & \quad 4 
\end{array}\right],$
which is positive definite. Thus, $F_1$ is a convex function of $p_{11}$ and $p_{12}$. Therefore, we can obtain an optimal solution by setting $\nabla F_1 = \mathbf{0}$, which corresponds to solving the following system of equations:
$$\begin{cases}
4 p_{11}- 2 =0, \\
4 p_{12} - 2=0. 
\end{cases} $$
We obtain the optimal location $\mathbf{p}_1^*=\left(\frac{1}{2},\frac{1}{2}\right)$, which is the center of the square. Note that if we apply Equation (\ref{oned}) derived for the one-dimensional case to the horizontal and vertical problems individually, we obtain the same optimal solution.

\subsection{Two hubs}
\label{TC}

In this subsection, we derive an analytical expression for the objective function in the case of two hubs, and determine an optimal solution by solving the corresponding nonlinear program. 

In the two-hub case, we want to determine the locations of $\mathbf{p}_1=(p_{11},p_{12})$ and $\mathbf{p}_2=(p_{21},p_{22})$ on the square. If we have two distinct points in a plane, one of them must be located either on the northeast side of the other one, or on the northwest side of the other one. We first assume that $\mathbf{p}_2$ is on the northeast side of $\mathbf{p}_1$, that is, $p_{21} \geq p_{11}$ and $p_{22} \geq p_{12}$, and, letting $\Delta_1=p_{21}-p_{11}$ and $\Delta_2=p_{22}-p_{12}$, that $\Delta_1 \geq \Delta_2$ (we later use symmetry to characterize the case in which $\mathbf{p}_2$ is on the northwest side and $\Delta_1<\Delta_2$). An illustration of such a situation is provided in Figure \ref{2p}.
 
\begin{figure}[htbp]
\centering
\includegraphics[width=0.5\textwidth]{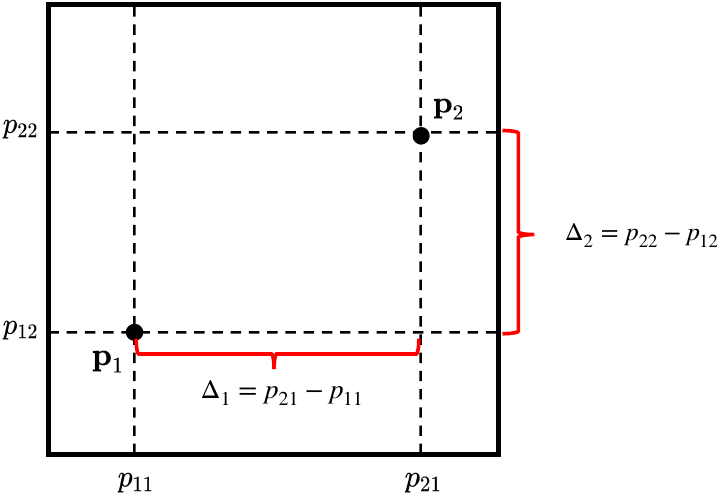}
\caption{Illustration of a Two-Hub Situation}
\label{2p}
\end{figure}

Correspondingly, in order to derive an analytical expression for the objective function, for each given customer location $\mathbf{x}$, we divide the square into a set of non-overlapping areas (excepting for boundary lines between areas) such that, for a service provider location $\mathbf{y}$ within each such area, we can identify one location (either $\mathbf{p}_1$ or $\mathbf{p}_2$) with $\min \{d_1,d_2\}$. Below, we present some details on how to characterize such areas. 

For any pair of $\mathbf{x}$ and $\mathbf{y}$, we say that 
$\mathbf{p}_1$ wins if $d_1 \leq d_2$ and $\mathbf{p}_2$ wins if $d_2 < d_1$. Note that for  ease of analysis and without loss of generality, if $d_1 = d_2$, we state that $\mathbf{p}_1$ wins. Based on the definitions of $d_1$ and $d_2$, we have that $d_1 \leq d_2$ is equivalent to 
\begin{equation}
\|\mathbf{y}-\mathbf{p}_1\|_{1} - \|\mathbf{y}-\mathbf{p}_2\|_{1} \leq  \|\mathbf{x}-\mathbf{p}_2\|_{1} - \|\mathbf{x}-\mathbf{p}_1\|_{1}. \label{Inequality}
\end{equation}

First, we focus on characterizing the right-hand side of Inequality (\ref{Inequality}), which represents the difference between the $\ell^1$ distance from $\mathbf{x}$ to $\mathbf{p}_2$ and the $\ell^1$ distance from $\mathbf{x}$ to $\mathbf{p}_1$. Since  $\mathbf{x}$ can be anywhere on the square, we define $\|\mathbf{x}-\mathbf{p}_2\|_{1} - \|\mathbf{x}-\mathbf{p}_1\|_{1} = b$ as a \textbf{constant distance difference line}, which includes all potential locations for $\mathbf{x}$ 
such that the difference in distance from $\mathbf{x}$ to $\mathbf{p}_2$ and from $\mathbf{x}$ to $\mathbf{p}_1$ equals $b$. If $b=0$, that is, $ \|\mathbf{x}-\mathbf{p}_1\|_{1} =  \|\mathbf{x}-\mathbf{p}_2\|_{1}$, we define such a special case of the constant distance difference line as the \textbf{balance line}.



\begin{figure}[htbp]
\centering
\includegraphics[width=0.3\textwidth]{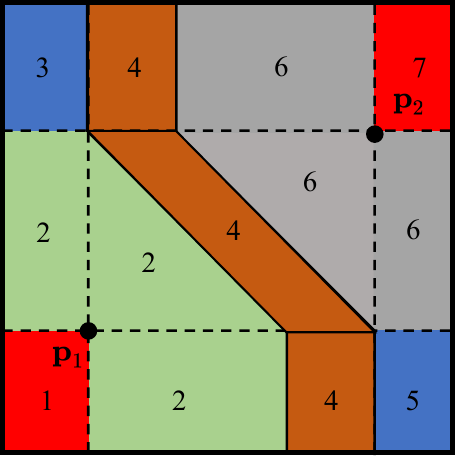}
\caption{Illustration of the Integration Area Division of $\mathbf{x}$}
\label{2p5}
\end{figure}

As shown in Figure \ref{2p5}, based on the locations of $\mathbf{p}_1$ and $\mathbf{p}_2$ and the associated $\Delta_1$ and $\Delta_2$ values, the square can be divided into seven areas using constant distance difference lines, for the location of $\mathbf{x}$. We observe that within each of the Areas 2, 4, and 6, every constant distance difference line consists of three line segments (illustrated by the dashed red lines in Figure \ref{2p6}). In Area 2, every line has horizontal, diagonal, and vertical segments in sequence, in Area 4, every line has vertical, diagonal, and vertical segments, while in Area 6, every line has vertical, diagonal, and horizontal segments. For the four corner areas, that is, Areas 1, 3, 5, and 7, the value of $b$ is the same for each point within the corresponding rectangle. Specifically, for Area 1, $b = \Delta_1+\Delta_2$, for Area 3, $b = \Delta_1-\Delta_2$, for Area 5, $b = \Delta_2-\Delta_1$, and for Area 7, $b = -\Delta_1-\Delta_2$. That is, constant distance difference lines become constant difference \emph{sets} consisting of each point in the rectangle corresponding to each of these four areas.

\begin{figure}[htbp]
\centering
\includegraphics[width=0.5\textwidth]{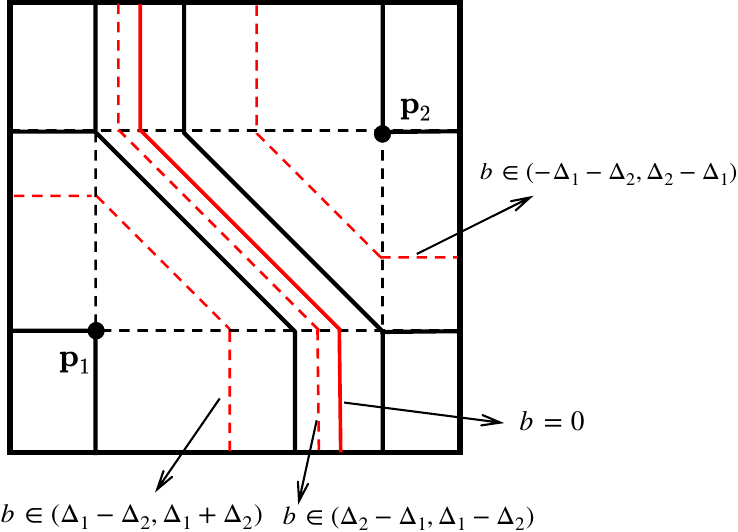}
\caption{Illustration of Constant Distance Difference Lines for Different $b$ Values}
\label{2p6}
\end{figure}

Similar to Corollary \ref{cor1} and Corollary \ref{cor2} introduced for the one-dimensional problem, we have the following two propositions, the proofs of which are in Appendix \ref{Appendix:Proposition}. 

\begin{proposition}
\label{2Dp1}
    If $\mathbf{x}$ is in Area 1, $d_1 = \min \{d_1, d_2\}$ regardless of the location of $\mathbf{y}$ in the square region.
\end{proposition}

\begin{proposition}
\label{2Dp2}
    If $\mathbf{x}$ is in Area 7, $d_2 = \min \{d_1, d_2\}$ regardless of the location of $\mathbf{y}$ in the square region.
\end{proposition}

Since the diagonal segment of the balance line (illustrated by the solid red line in Figure \ref{2p6}) passes through the point $\left(\frac{p_{11}+p_{21}}{2},\frac{p_{12}+p_{22}}{2}\right)$ and has a slope of $-1$, we can derive its equation as
\begin{equation}
    2x_1 + 2x_2 = p_{11} + p_{12} + p_{21} + p_{22}. \label{edy}
\end{equation}
For any constant distance difference line with $b \in (-\Delta_1-\Delta_2,\Delta_1+\Delta_2)$, that is, with $\|\mathbf{x}-\mathbf{p}_2\|_{1} - \|\mathbf{x}-\mathbf{p}_1\|_{1} = b$, we can write the equation for the corresponding diagonal segment as
\begin{equation}
2x_1 + 2x_2 + b = p_{11}+p_{12}+p_{21}+p_{22}. \label{xddl}
\end{equation}

As shown in Figure \ref{div}, given an $\mathbf{x}$ located in the area corresponding to the label below each subfigure (for Areas 2 through 7), we can define an associated \textbf{separation line} (shown as the boundary lines in Figure \ref{div}) such that if $\mathbf{y}$ lies in the red area then 
$\|\mathbf{y}-\mathbf{p}_1\|_{1} - \|\mathbf{y}-\mathbf{p}_2\|_{1} \leq b$ (and $\mathbf{p}_1$ wins), whereas if $\mathbf{y}$ lies in the blue area then $\|\mathbf{y}-\mathbf{p}_1\|_{1} - \|\mathbf{y}-\mathbf{p}_2\|_{1} > b$ (and $\mathbf{p}_2$ wins). Accordingly, for any $b \in (-\Delta_1-\Delta_2,\Delta_1+\Delta_2),$ each associated separation line (shown as the boundary lines for Areas 2 through 6 in Figure \ref{div}) also has three segments, the middle of which is diagonal. For a given $b$, the corresponding diagonal segment's equation is 
\begin{equation}
    2y_1 + 2y_2 - b = p_{11}+p_{12}+p_{21}+p_{22}.\label{yddl}
\end{equation}

\begin{figure}[htbp]
\centering
\includegraphics[width=0.8\textwidth]{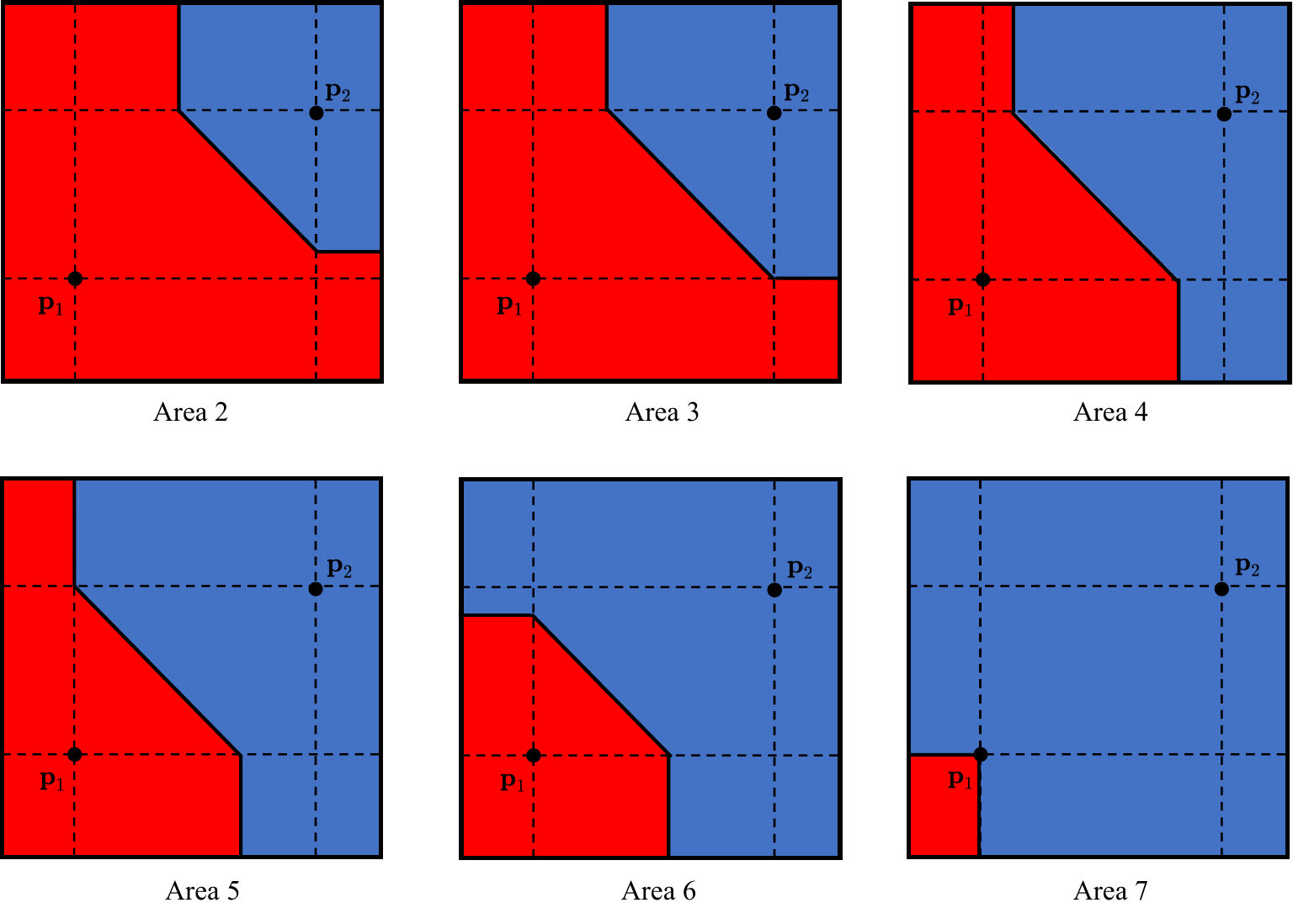}
\caption{Illustration of the Integration Area of $\mathbf{y}$ Given $\mathbf{x}$ in Area 2 to Area 7}
\label{div}
\end{figure}

The diagonal segment of the separation line (\ref{yddl}) and the diagonal segment of the constant distance difference line (\ref{xddl}) are symmetric about the balance line (\ref{edy}), and each of these has a slope of $-1$. As a result, given that $\mathbf{x}$ lies within each of the areas shown in Figure \ref{2p5}, we can explicitly obtain the areas of integration for $\mathbf{y}$ and the corresponding distance minimizing hub for each area (recall that Proposition \ref{2Dp1} implies that $\mathbf{p}_1$ always wins when $\mathbf{x}$ is in Area 1). Based on a detailed area-by-area integration procedure, we obtain the objective function denoted as $F_2$ (all details are provided in Appendix \ref{Appendix:OBJ Expression of Two hubs}).

The corresponding nonlinear programming model is as follows:
\begin{align*}
\min \quad & F_2 \\
\mbox{s.t.} \quad &p_{21} \geq p_{11}, \\
    &p_{22} \geq p_{12}, \\
    &p_{21}-p_{11} \geq p_{22}-p_{12}, \\
    &0 \leq p_{11}, p_{12}, p_{21}, p_{22} \leq 1.
\end{align*} 
We observe that the objective function is a multi-variate polynomial function of degree 5 that is not convex with respect to the decision variables  $p_{11}$, $p_{12}$, $p_{21}$, and $p_{22}$. The optimal solution obtained using the global optimization solver BARON (\citet{sahinidis1996baron}) is $\mathbf{p}_1^*=(0.3237,0.3650)$ and $\mathbf{p}_2^*=(0.6763,0.6350)$, and the corresponding objective function value is 0.8746. Since $0.3237 + 0.6763 = 1$ and $0.3650 + 0.6350 =1$, we observe that $\mathbf{p}_1^*$ and $\mathbf{p}_2^*$ are symmetric with respect to the center of the square. Moreover, they are off the bottom-left to top-right diagonal of the square but lie close to it.
Similarly, following the same procedure, we can obtain three alternative optimal solutions, that is, $\mathbf{p}_1^*=(0.3650,0.3237)$ and $\mathbf{p}_2^*=(0.6350,0.6763)$ when $\Delta_1 < \Delta_2$, and 
$\mathbf{p}_1^*=(0.6763,0.3650)$ and $\mathbf{p}_2^*=(0.3237,0.6350)$ and  $\mathbf{p}_1^*=(0.6350,0.3237)$ and $\mathbf{p}_2^*=(0.3650,0.6763)$ for the cases where $\mathbf{p}_2$ is to the northwest of $\mathbf{p}_1$. 

It is worth noting that if we assume that $p_{12}=p_{22}$, then the expression for $F_2$ simplifies to 
\begin{align*}
    F_2 &=2+p_{11}^{2} p_{21}-p_{11} p_{21}^{2}+\frac{1}{3} p_{11}^{3}\\
    &-\frac{1}{3} p_{21}^{3}-2 p_{21}+2 p_{21}^{2}+2 p_{22}^{2}-2 p_{22}. 
\end{align*}
The optimal solution for this special case is $\mathbf{p}_1^*=\left(\frac{2-\sqrt{2}}{2},\frac{1}{2}\right)$ and $\mathbf{p}_2^*=\left(\frac{\sqrt{2}}{2},\frac{1}{2}\right)$, which is consistent with the two-dimensional result in \citet{campbell1990freight}'s lattice setting. 

\section{Two-Dimensional $n$-Hub Problem}
\label{sec:Simulation approximation}

Because it becomes extremely difficult to extend the procedure proposed for the two-hub setting to problems with three or more hubs, in this section, we propose a simulation-based approximation method to handle the two-dimensional $n$-hub problem. Our approximation method consists of two steps. First, we discretize possible customer, service provider, and  hub locations using grids of different sizes, and construct a large-scale discrete $p$-median problem. Using Benders decomposition, we can quickly obtain an initial solution, which serves as input for the second step. Second, we leverage simulation optimization to continually improve the performance of the solution in the continuous space. The following subsections provide details of our implementation and experimental results.

\subsection{Discrete $p$-Median Problem}

A unique dimension of our problem analysis lies in the continuity of the service provider and customer location spaces and the solution space (i.e., the locations of hubs). Furthermore, a non-trivial component of our problem is in characterizing and computing the expected value in the objective function,  which involves the integration over continuous spaces. To address this, we employ a numerical integration technique (i.e., quadrature) to approximate the objective function. Specifically, we discretize the continuous space into a grid with a total of $S$ points, where each point represents a potential location for a random service provider or a random customer. Assuming a uniform distribution setting, the probability of the provider or customer being at each point is $1/S$. Consequently, we obtain a total of $J=S^2$ customer and service provider demand pairs. Likewise, we discretize the continuous solution space into a grid containing $K$ points, such that each point serves as a potential location for one of the $n$ hubs. 

\subsubsection{Classical Formulation}
\label{classicalformulation}
For each demand pair and potential hub location, we can calculate a corresponding $\ell^1$ travel distance as defined by $ f_{jk}, j\in[J],k\in[K]$. In addition, we define $v_k$ as a binary decision variable, which is equal to 1 if a hub is located at point $k$ and 0 otherwise. Similarly, we define $u_{jk}$ as a binary decision variable, which is equal to 1 if a facility at point $k$ is utilized for demand pair $j$ and 0 otherwise. Consequently, the discrete version of our continuous $n$-hub location problem can be formulated as the following binary integer program:
\begin{align}
\min \quad & \frac{1}{J}\sum^J_{j=1}\sum^K_{k=1}f_{jk}u_{jk}\label{pmedian_obj}\\ 
\text { s.t.} \quad & \sum^K_{k=1}v_k=n, \label{pmedian_1} \\
& \sum^K_{k=1}u_{jk}=1, \quad j\in [J], \label{pmedian_2}\\
& u_{jk}\le v_k, \quad j\in [J], k\in [K], \label{pmedian_3}\\
& u_{jk} \in \{0,1\}, \quad j\in [J], k\in [K], \label{pmedian_4}\\
& v_k\in \{0,1\}, \quad  k\in [K]. 
\end{align}
Constraint (\ref{pmedian_1}) states that a total of $n$ hubs will be located. Constraint set (\ref{pmedian_2}) ensures that exactly one facility will be selected for each demand pair. Constraint set (\ref{pmedian_3}) guarantees that if a hub is not located at point $k$, that is, $v_{k}=0$, no demand pair can utilize a facility at point $k$. Finally, Constraint set (\ref{pmedian_4}) can be relaxed to $ 0\le u_{jk}\le 1, j\in [J], k\in [K]$ without loss of optimality. The associated optimization model is exactly the classical $p$-median formulation by \cite{revelle1970central} if we rescale each $f_{jk}$ by $\frac{1}{J}$. Since the value of $J$ is large even for a small-size grid, our problem is equivalent to a large-scale $p$-median problem after discretization. 

\subsubsection{Elloumi’s Formulation}
\label{SUB:Elloumi}
The grids used for discretization contain regular shapes, and the demand pairs follow an organized pattern. As a result, the travel distance multiset $T_j = \{f_{jk}: k\in [K]\},  j\in [J] $, which encodes the travel distance between demand pair $j$ via any potential facility location will contain many duplicate values. That is, given a demand pair, numerous candidate facility locations share the same travel distance.  To leverage this property, we adopt the $p$-median formulation proposed by \cite{elloumi2010tighter},  which results in a significant reduction in the number of decision variables when compared to the classical formulation in Section \ref{classicalformulation}.

In Elloumi's formulation, instead of considering $f_{jk}$ for each demand pair $j$ and potential hub location $k$ in the objective function, the formulation only considers distinct distances for each demand pair $j$ . Let $R_j$ be the number of distinct distances in $T_j$, and let $D^1_j < D^2_j < \cdots < D^{R_j}_j$ be these sorted distances. Additionally, a binary decision variable $z_j^r, j\in[J], r\in [R_j]$ is introduced to track the travel distance associated with demand pair $j$, replacing $u_{jk}$ in the classical formulation. The variable $z_j^r$ is set to 0  if and only if there exists an assigned hub location $k$ with $f_{jk} \le D^{r}_j$.  The location variable $v_k$ remains the same as in the classical formulation. As a result, Elloumi's formulation can be represented as
\begin{align}
    \min \; &\frac{1}{J}\sum^J_{j=1}\left(D^1_j+ \sum^{R_j-1}_{r=1} \left(D^{r+1}_j-D^r_j\right)z^r_j\right)\\
    \text{s.t. } \;&\sum^K_{k=1}v_k=n, \label{Elloum_1} \\
    & z^1_j+\sum_{k:f_{jk}=D^1_j}v_k \ge 1, \quad j\in [J], \label{Elloum_2}\\ 
    & z^r_j+\sum_{k:f_{jk}=D^r_j}v_k \ge z^{r-1}_j, j\in [J], r\in [R_j]\backslash\{1\},\label{Elloum_3} \\
    & z^r_j \ge 0,\quad j\in [J], r \in [R_j],  \label{Elloum_4} \\
    & v_k\in \{0,1\}, \quad k\in [K].
\end{align}
Constraint (\ref{Elloum_1}) is the same as Constraint (\ref{pmedian_1}). Constraint sets (\ref{Elloum_2}) and (\ref{Elloum_3}) ensure that if no assigned hub location $k$ with $f_{jk} \le D^{r}_j$ exists, $z^r_j$ must be set to 1.  Constraint set (\ref{Elloum_4}) indicates that $z^r_j$ can be relaxed and treated as a continuous variable without loss of optimality. 

\subsubsection{Benders Decomposition}

We apply Benders decomposition \citep{bnnobrs1962partitioning,duran2023efficient} to solve the large-scale $p$-median problem efficiently, as it has demonstrated superior performance in solving facility location problems \citep{cornuejols1980canonical,magnanti1981accelerating, fischetti2017redesigning}. This is attributed to the closed-form Benders cuts generated by the corresponding sub-problems.


At iteration $\tau$, the formulation of the master problem is:
\begin{align*}
    \min \quad & \frac{1}{J}\sum^J_{j=1} \eta_j \\\text { s.t.} \quad & \sum^K_{k=1}v_k=n,\\
     & v_k\in \{0,1\}, \quad k\in [K], \\
     & BC_{j}^{l},  \quad j\in [J], l \in [\tau],
\end{align*}
where $BC_{j}^{l}$ corresponds to the Benders cut generated by the $l^{\rm{th}}$ Benders subproblem solution.  According to \cite{duran2023efficient}, the Benders cut, $BC_{j}^{\tau+1}$, has the following closed form:
$$
    BC_{j}^{\tau+1}=\begin{cases}\eta_j \ge D^1_j, & \text { if } \tilde r_j^{\tau+1} =0, \\
    \eta_j \ge D^{\tilde r_j^{\tau+1}+1}_j-\sum_{k:f_{jk}\le D_j^{\tilde r_j^{\tau+1}}}\left(D^{\tilde r_j^{\tau+1}+1}_j-f_{jk}\right)v_k, & \text { otherwise,} \end{cases}
$$
where $\tilde r_j^{\tau+1}$ is calculated via the relaxed master problem's solution $\bar {\mathbf{v}}^{\tau}$ as
$$
\tilde r_j^{\tau+1}= \begin{cases}0, & \text { if } \sum_{k:f_{jk}=D^1_j}\bar v_k^{\tau} \ge 1, \\ \max\left\{ r\in [R_j]:\sum_{k:f_{jk}\le D^r_j}\bar v_k^{\tau} < 1\right\}, & \text { otherwise.} \end{cases}
$$

The scalability of our discretization method is an important feature, as it allows adjusting the grid size to balance approximation quality and computation speed. We leverage this scalability by solving the problem on customer and provider grids with fewer points as a heuristic to obtain a high-quality initial solution. This initial solution, in turn, serves as the basis for generating initial Benders cuts, leading to a substantial reduction in the number of iterations and enhanced computational efficiency.

\subsubsection{Model Reduction}

We mentioned the reduction of the number of required decision variables using Elloumi’s formulation in Section \ref{SUB:Elloumi}. Specifically, the total number of $\mathbf{u}$ variables is $J\times K$ in the classical formulation. In contrast, the total number of $\mathbf{z}$ variables in Elloumi’s formulation is $\sum^J_{j=1}R_j$, which is significantly smaller than $J\times K$. We can further reduce the model size by decreasing the number of demand pairs. The current total number of demand pairs is $S^2$. By exploiting the symmetric structure of these demand pairs, the number of demand pairs can be reduced to $S(S+1)/2$. To illustrate the significance of model reduction, suppose we use a $10\times10$ grid with 121 points for the provider and customer location spaces and a $20\times20$ grid with 441 points for the solution space. The total number of demand pairs is 14,641. After the demand pair reduction, the number is reduced to 7,381. With the reduced number of demand pairs, the classical formulation has 3,255,021 $\mathbf{u}$ variables, whereas Elloumi's formulation has 155,661 $\mathbf{z}$ variables, resulting in a reduction of more than $95\%$.

\subsection{Simulation Optimization}

Simulation optimization typically refers to optimization in a setting where the objective function cannot be computed exactly \citep{fu2015handbook}. The expectation in the objective function of our problem aligns with this setting. In other words, obtaining a closed-form expression for the objective function is no longer straightforward, and the expression becomes intractable quickly as the number of hubs increases. Fortunately, if a feasible set of hub locations is given, the objective function associated with the feasible solution can be approximately evaluated by simulation in a trivial way, by averaging the sum over the minimum distances between hubs and sampled demand pairs. 

Because our problem is a high-dimensional optimization problem, where the dimensionality of the optimization problem is twice the number of hubs, we use simultaneous perturbation stochastic approximation (SPSA) \citep{spall1992multivariate} as the simulation optimization method.  Unlike traditional gradient-based methods that require evaluating the gradient of the objective function with respect to each decision variable, SPSA approximates the gradient using only two measurements of the objective function per iteration, regardless of the problem's dimensionality.

SPSA is a black box algorithm that requires fine-tuning of a set of hyperparameters. Although \cite{spall1998implementation} offers systematic procedures for determining these parameters, users are still required to set the ``step" parameter beforehand, representing the initial desired magnitude of change in the decision variable values. The grid size of the solution space set in the discretization step serves as a useful guide for setting up the step parameter, as we discuss in the following section. Moreover, like many other simulation optimization methods, SPSA has the potential to become trapped in local optima. This limitation is mitigated by using the solution to the discretized problem as an initial starting point, which is itself typically a high-quality solution.

\subsection{Computational Experiments}

The experiments were conducted using an Apple M2 pro processor with a clock speed of 3.5 GHz and 16 GB of RAM. The code is developed in Python 3.7, utilizing Gurobi 11.0. 

\subsubsection{Discretization}

\label{discrete_subs} 

We discretize the provider and customer spaces using a $10\times10$ grid and the continuous solution space using a $20\times20$ grid, where each intersection point serves as a potential location. As a result, the large-scale $p$-median problem contains 7,381 demand pairs (after demand reduction) and 441 potential hub locations. Moreover, we solve a smaller problem with a $5\times5$ grid for the provider and customer spaces and a $20\times20$ solution space grid using Gurobi directly to provide an initial solution for the Benders decomposition algorithm. The classical and Elloumi's formulations are directly solved using Gurobi's default mixed-integer solver, and the Benders decomposition algorithm is implemented by a Branch-and-Benders-Cut framework with callback functions. A comparison of computation times among the classical formulation, Elloumi's formulation, and Benders decomposition (Elloumi's formulation with Benders decomposition), as shown in Table \ref{result}, reveals significant time savings with our approach. 

\begin{table}[h!]
\centering
\caption{Comparison of Computation Times (sec)}
\label{result}
\begin{tabular}{@{}lccc@{}}
\toprule
$n$ & Classical Formulation & Elloumi's Formulation & Benders Decomposition \\ 
\midrule
2  & 287   & 35   & 21  \\
3  & 931   & 48   & 24  \\
4  & 31    & 26   & 6   \\
5  & 996   & 143  & 17  \\
6  & 4,329 & 236  & 64  \\
7  & 3,324 & 234  & 59  \\
8  & 3,407 & 183  & 82  \\
9  & 1,453 & 123  & 20  \\
10 & 9,259 & 501  & 231 \\ 
\bottomrule
\end{tabular}
\end{table}


It is worth noting that most cases with a large value of $n$ can be efficiently solved by Benders decomposition within a few iterations, largely attributable to the fact that the $p$-median problem is known to be ``Integer Friendly" \citep{morris1978extent} because the gap between the linear programming (LP) relaxation solution and the optimal integer solution is often quite small or nonexistent. We conduct experiments for $n \in \{20,25,30,\ldots, 200\}$ on a $20\times20$ solution space grid, and for $n \in \{50,55,60,\ldots, 300\}$ on a $50\times50$ solution space grid. As shown in Figure \ref{Large_p}, for large $n$ values, the problems can be solved rapidly, with the majority of cases yielding optimal integer solutions by solving their LP relaxations.


\begin{figure}[htbp]
    \centering
    \begin{subfigure}{0.45\textwidth}
        \centering
        \includegraphics[width=\linewidth]{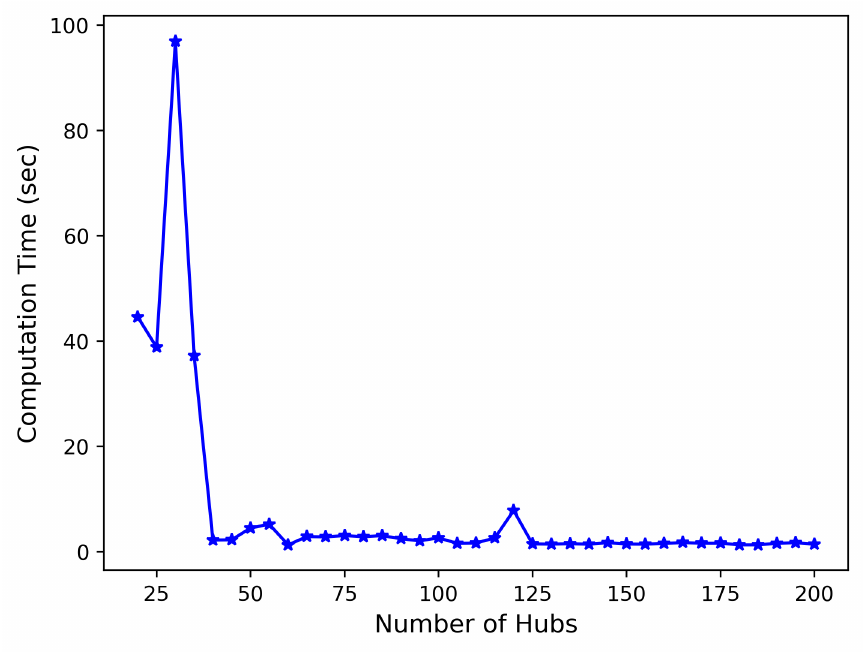}
        \caption{$20\times20$ Solution Space Grid}
    \end{subfigure}
    \hspace{1em}
    \begin{subfigure}{0.45\textwidth}
        \centering
        \includegraphics[width=\linewidth]{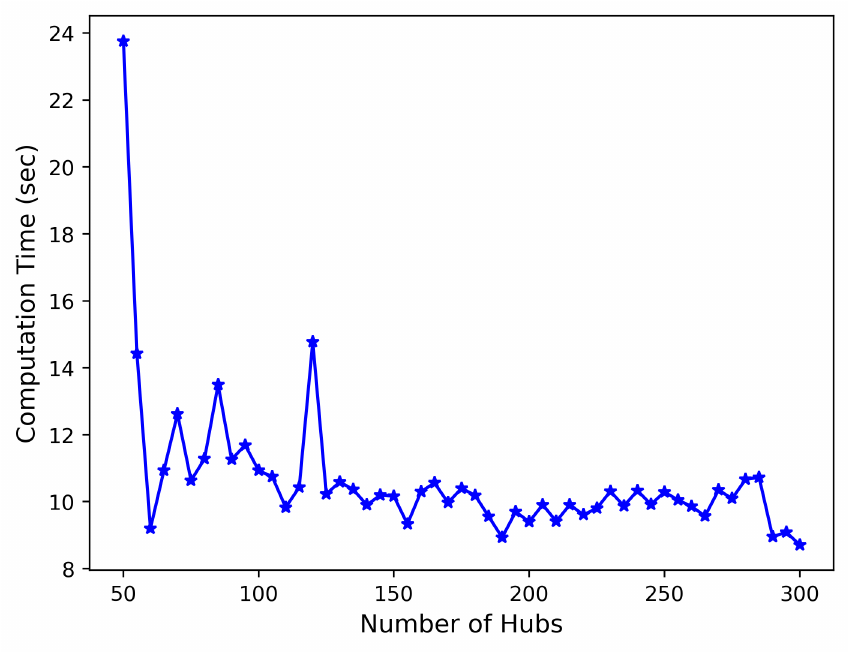}
        \caption{$50\times50$ Solution Space Grid}
    \end{subfigure}
    \caption{Benders Decomposition Computation Times for Large $n$ Values}
    \label{Large_p}
\end{figure}

\subsubsection{SPSA}

\label{SPSAsetting}

In the SPSA experimental settings, we use three sample sizes (i.e., 5,000, 10,000, and 100,000) with samples generated using the Monte Carlo method for objective function evaluation in each SPSA iteration. We tested the continuous problem instances with $n = 2, 3, \ldots, 10$. Although the discretized solution method with Benders decomposition can generate initial instances with larger values of $n$, the required solution space and time increase quickly as the number of hubs increases. To ensure statistical significance and mitigate the impact of randomness, each experiment is repeated 20 times. Only the solutions that have converged are considered for further analysis. Here, we define \emph{unconverged} solutions as those such that 1) an iteration limit of 500 is reached; 2) the objective function value (estimated by the specific sample size) of the terminated solution is greater than the initial solution value; or 3) the terminated solution is outside the square. Finally, the true objective function values of solutions are estimated using a large sample size of 10,000,000.

All hyperparameters except the step parameter are fine-tuned using Spall's suggested methods \citep{spall1998implementation}. For the step parameter, since the resolution of the initial solution is 0.05 for the $20\times20$ solution space grid, we can use a value smaller than this to set the step parameter, e.g., 0.02. While this approach typically suffices, in scenarios where the initial solution is very close to the optimal solution, a smaller value, such as 0.005, is recommended to prevent solutions from not converging. In such cases, the computation time does not significantly increase due to the initial high-quality solution.

In order to demonstrate the superiority and necessity of the discretization step (we refer to this as the Benders approach), we also conducted an experiment on initial solution generation by utilizing a uniform distribution to generate locations on a square (we refer to this as the Random approach). In contrast to the starting solution from the Benders approach, where we can use the distance between grid points as a guide for determining an upper bound on the step size, no such grid points are available for bounding the step parameter under the Random approach. Hence, we fine-tune the step size by employing a grid search (e.g., 0.1 for $n=2$ to  $n=6$ and 0.05 for the rest). Other hyperparameters are also fine-tuned by Spall's methods. The experimental results corresponding to the mean computation time (for converged instances), mean objective value difference, and number of instances without convergence are depicted in Figures \ref{CT} and \ref{CMOV}.  It is important to note that all of the solutions converged under the Benders approach, while this is not the case under the Random approach, as seen in Figure \ref{CMOV:b}.


\begin{figure}[htbp]
    \centering
    \begin{subfigure}{0.3\textwidth}
        \centering
        \includegraphics[width=\linewidth]{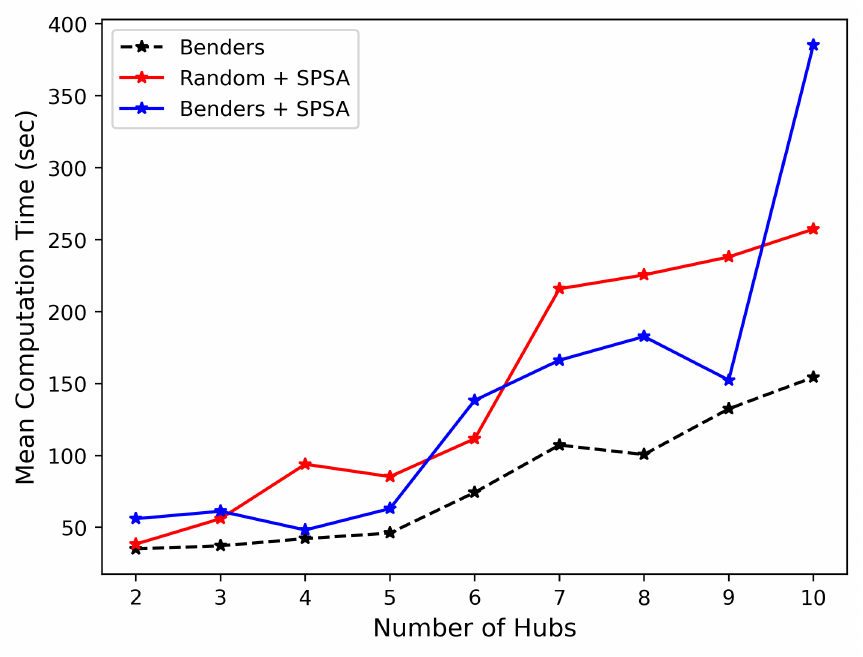}
        \caption{ }
    \end{subfigure}
    \hspace{1em}
    \begin{subfigure}{0.3\textwidth}
        \centering
        \includegraphics[width=\linewidth]{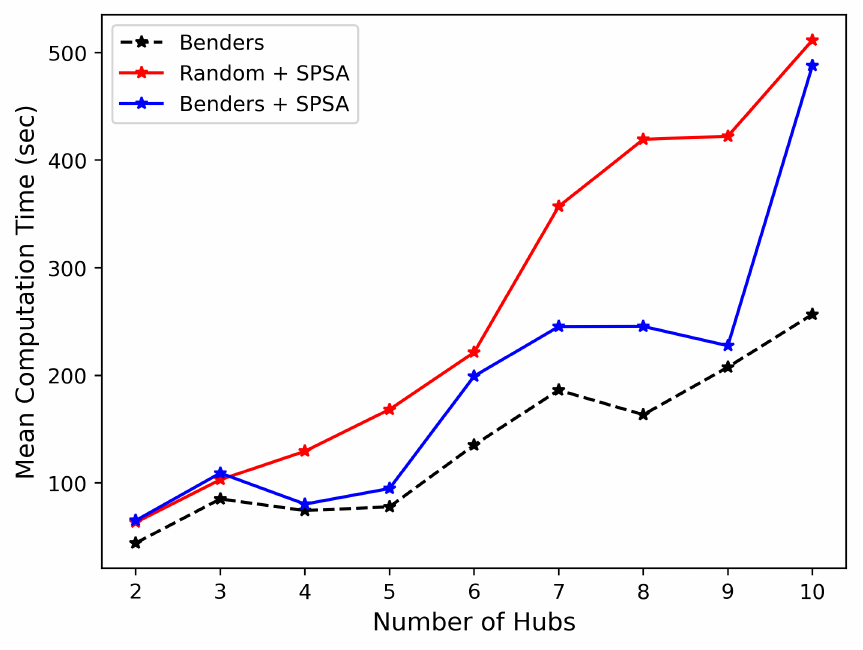}
        \caption{ }
        \label{CT_1000}
    \end{subfigure}
    \hspace{1em}
    \begin{subfigure}{0.3\textwidth}
        \centering
        \includegraphics[width=\linewidth]{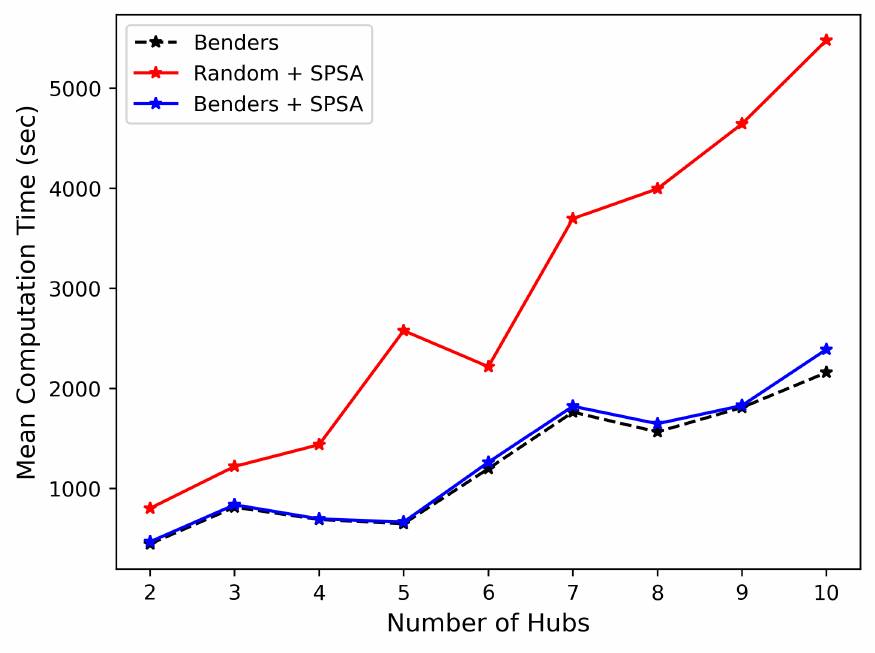}
        \caption{ }
        \label{CT_10000}
    \end{subfigure}
    \caption{Comparison of Mean Computation Times: (a) Sample Size 5,000, (b) Sample Size 10,000, and (c) Sample Size 100,000 }
    \label{CT}
\end{figure}


\begin{figure}[htbp]
    \centering
    \begin{subfigure}{0.45\textwidth}
        \centering
        \includegraphics[width=\linewidth]{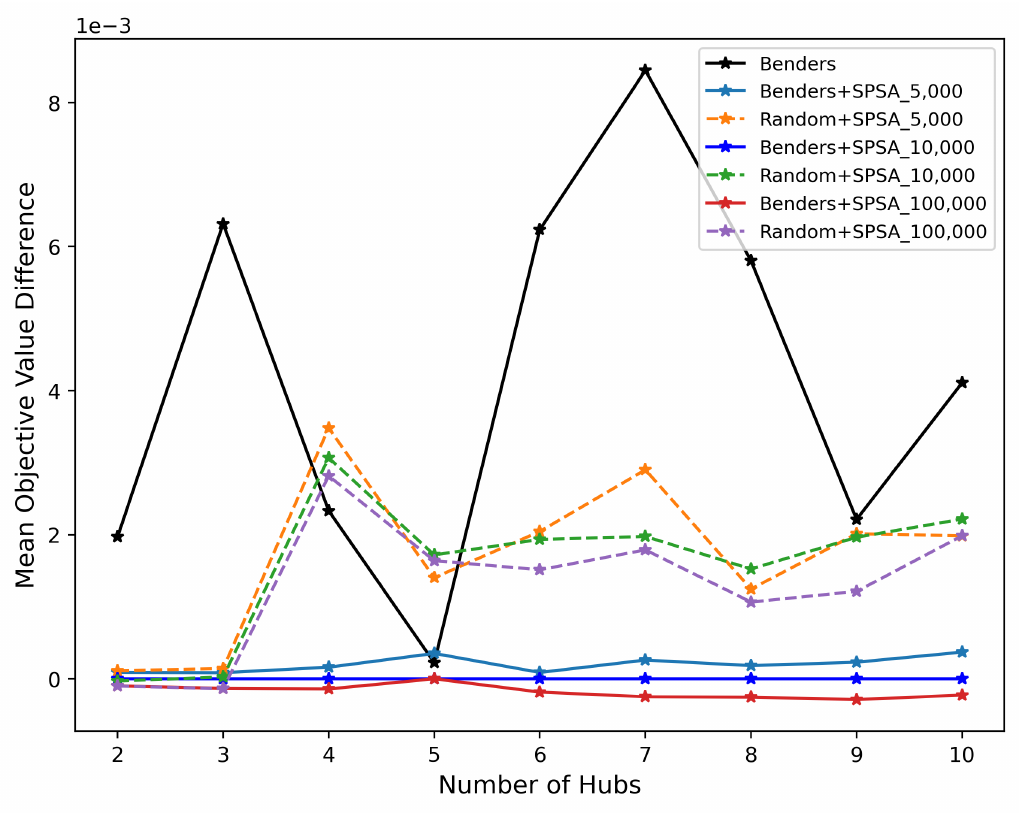}
        \caption{ }
        \label{CMOV:a}
    \end{subfigure}
    \hspace{1em}
    \begin{subfigure}{0.45\textwidth}
        \centering
        \includegraphics[width=\linewidth]{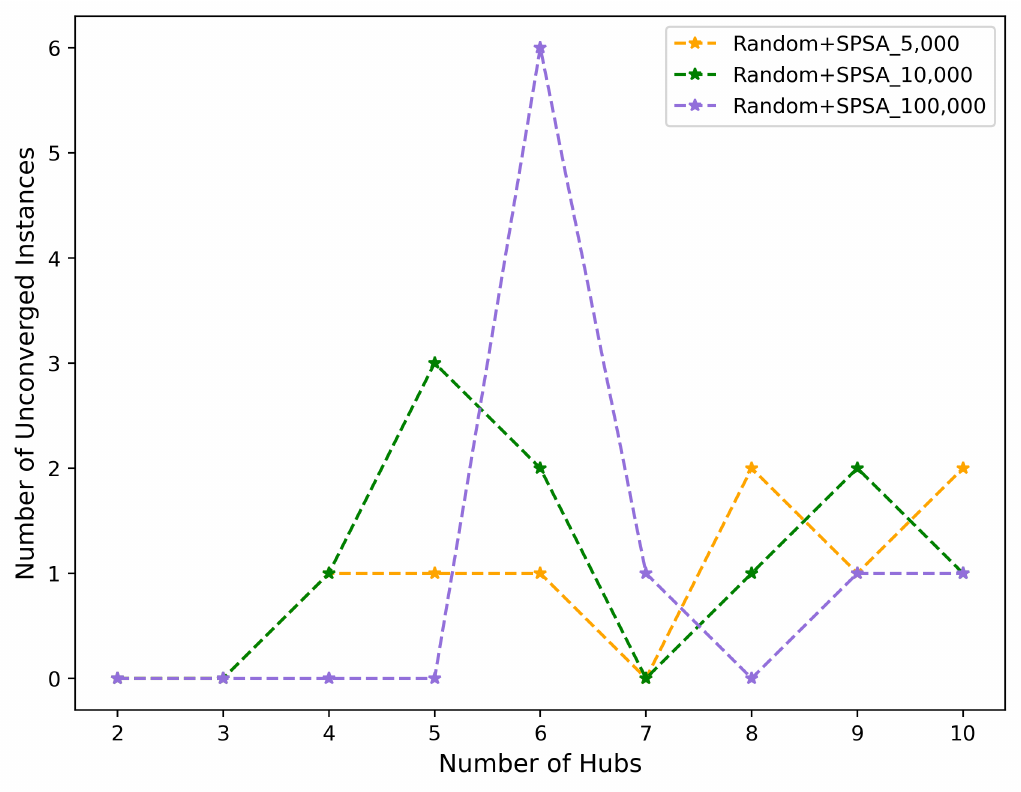}
        \caption{ }
        \label{CMOV:b}
    \end{subfigure}
    \caption{(a) Comparison of Mean Objective Value Differences Relative to the Base (Benders + SPSA\_10,000) and (b) Comparison of the Number of Unconverged Instances with 20 Repeated Trials}
    \label{CMOV}
\end{figure}

Based on our computational results, SPSA can significantly improve solution performance over the initial solution from the Benders approach. Meanwhile, it is evident that dedicating time to the discretization step is highly worthwhile. Firstly, the Benders approach yields a high-quality initial solution, effectively reducing the number of iterations and ultimately saving significant computation time, particularly when the sample size increases, as shown in Figure \ref{CT_1000} and \ref{CT_10000}. Secondly, leveraging the Benders approach results in the mean objective values of the final solutions from SPSA outperforming those of the Random approach (see Figure \ref{CMOV:a}); even a sample size of 5,000 under the Benders approach surpasses Random’s 100,000 sample size. The law of large numbers (LLN) aligns with our experimental results under the Benders approach, suggesting that a larger sample size leads to a higher-quality solution. It is worth noting that when $n = 5$, the result of Benders’ 5,000 sample size has a higher objective function value than the initial solution.  Smaller sample sizes lead to higher objective function approximation errors, which can be remedied by using larger sample sizes.   Therefore, we recommend using the largest possible sample size in practical applications to minimize the impact of approximation errors on the solutions. On the contrary, the LLN is not reflected in the results under the random approach (e.g., $n =$ 5, 8, and 10) due to the uncertainty in the quality of the random initial solutions. Additionally, such uncertainty also explains why solutions that do not converge could not be entirely avoided with the Random approach (see Figure \ref{CMOV:b}). 

In summary, the results from the Benders approach exhibit better quality and stability compared to those achieved from the Random approach. These findings clearly highlight the effectiveness of allocating resources to the discretization step, ultimately yielding superior solution quality and stability in the results. The final location results by ``Benders+SPSA-100,000" are shown in Figure \ref{finallocation}.

\begin{figure}[htbp]
\centering
\includegraphics[width=1\textwidth]{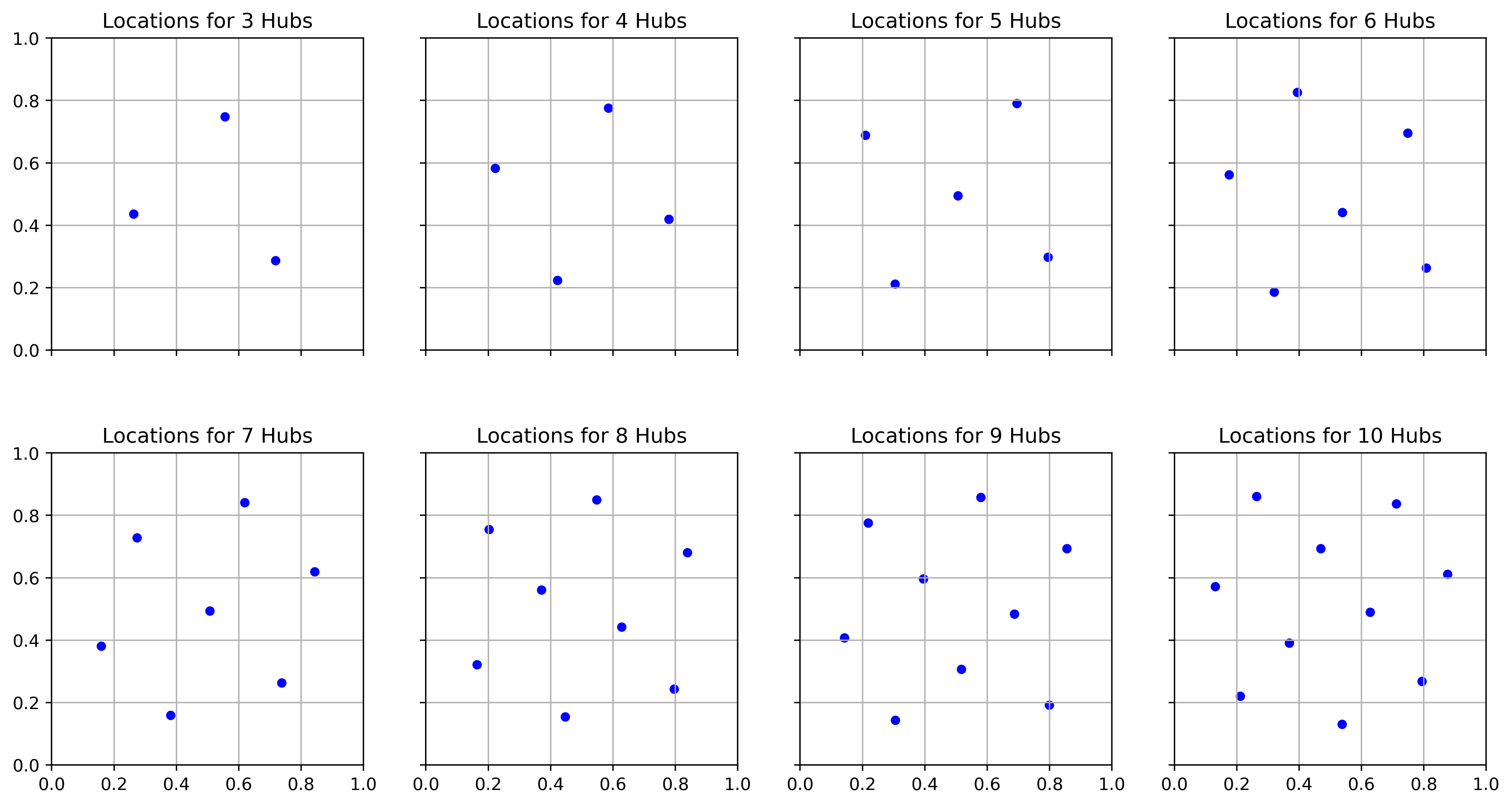}
\caption{Final Location Results for $n=3,\ldots,10$
}
\label{finallocation}
\end{figure}

\section{Two-Dimensional Multiple-Service-Provider Setting}

\label{sec:multi_service}
In real-life scenarios, it is common for multiple service providers to be available for a single customer rather than a one-to-one provider-customer relationship. For instance, when a customer requests a food delivery, multiple drivers may be simultaneously available on the map. However, the platform's assignment system typically selects only one driver with the shortest delivery distance. Therefore, for practical applicability, we extend our problem to encompass the setting of multiple service providers.

Suppose we have $W$ providers simultaneously. Let $\mathbf{X}$ denote the location of a random customer and $\mathbf{Y}_w, w\in [W]$, denote the locations of $W$ random service providers. Then, the extended problem becomes
\begin{equation*}    \min_{\mathbf{p}_1,\ldots,\mathbf{p}_n}E\left[\min_{i \in [n] ,w \in [W]}\Big\{\|\mathbf{X}-\mathbf{p}_i\|_1+\|\mathbf{Y}_w-\mathbf{p}_i\|_1\Big\}\right].
\end{equation*}

When we apply quadrature to approximate the extended problem, the ``Curse of Dimensionality" becomes a major issue. For instance, if we use a $10\times10$ grid for the continuous customer and provider spaces as in Section \ref{discrete_subs}, the total number of demand pairs constructed before reduction is $11^{2(W+1)}$. The approximation model is intractable even when $W=2$. This issue aligns with the fact that quadrature methods have precision of $O(\alpha^{-1/\beta})$ \citep{dunn2022exploring}, where $\alpha$ is the sample size and $\beta$ is the dimension of the integral. Therefore, to alleviate the Curse of Dimensionality, we opt for Monte Carlo integration over quadrature to approximate the expectation. The approximation error for Monte Carlo is $O(\alpha^{-1/2})$ \citep{dunn2022exploring}, which is independent of the dimension and thus well-suited for higher-dimensional problems.

We notice that the discrete version of the extended problem approximated by the Monte Carlo method can also be reformulated as a $p$-median problem, which enables our proposed simulation-based approximation method. The reasoning is as follows. If we have generated $J'$ samples (indexed by $j$), that is, $\mathbf{x}_j,\mathbf{y}_{j1},\ldots,\mathbf{y}_{jW}$, the expected objective function of the extended problem can be approximated by 
\begin{equation*}
    \frac{1}{J'}\sum_{j=1}^{J'}\ \min_{i \in [n],w\in [W]}\Big\{\|\mathbf{x}_j-\mathbf{p}_i\|_1+\|\mathbf{y}_{jw}-\mathbf{p}_i\|_1\Big\}.
    \nonumber
\end{equation*}
Let $g_{ji} =\min_{w\in[W]}\Big\{\|\mathbf{x}_j-\mathbf{p}_i\|_1+\|\mathbf{y}_{jw}-\mathbf{p}_i\|_1\Big\}$. The approximation of the extended problem can be reformulated as 
\begin{equation*}
\min_{\mathbf{p}_1,\ldots,\mathbf{p}_n}\frac{1}{J'}\sum_{j=1}^{J'}\min_{i \in[n]}g_{ji}.
\end{equation*}
Therefore, the discrete version of our approximated extended problem can be reformulated as a $p$-median problem.

\section{Case Study}
\label{sec:Case Study}
Our case study focuses on a $0.7 \times 0.7$ square mile region of Virginia Beach, highlighted by a red square on the map in Figure \ref{case1}. This area, a major tourist destination in Virginia Beach, attracts a significant number of visitors each year, resulting in an elevated OHCA incidence. The size of the selected area is well-suited for studying AED delivery via smartphone apps, as these apps are designed to fill the critical time gap before an ambulance arrives. Additionally, this region's commercial and tourism characteristics align with the Manhattan metric, making it an ideal distance measurement for our analysis.

\begin{figure}[htbp]
\centering
\includegraphics[width=1\textwidth]{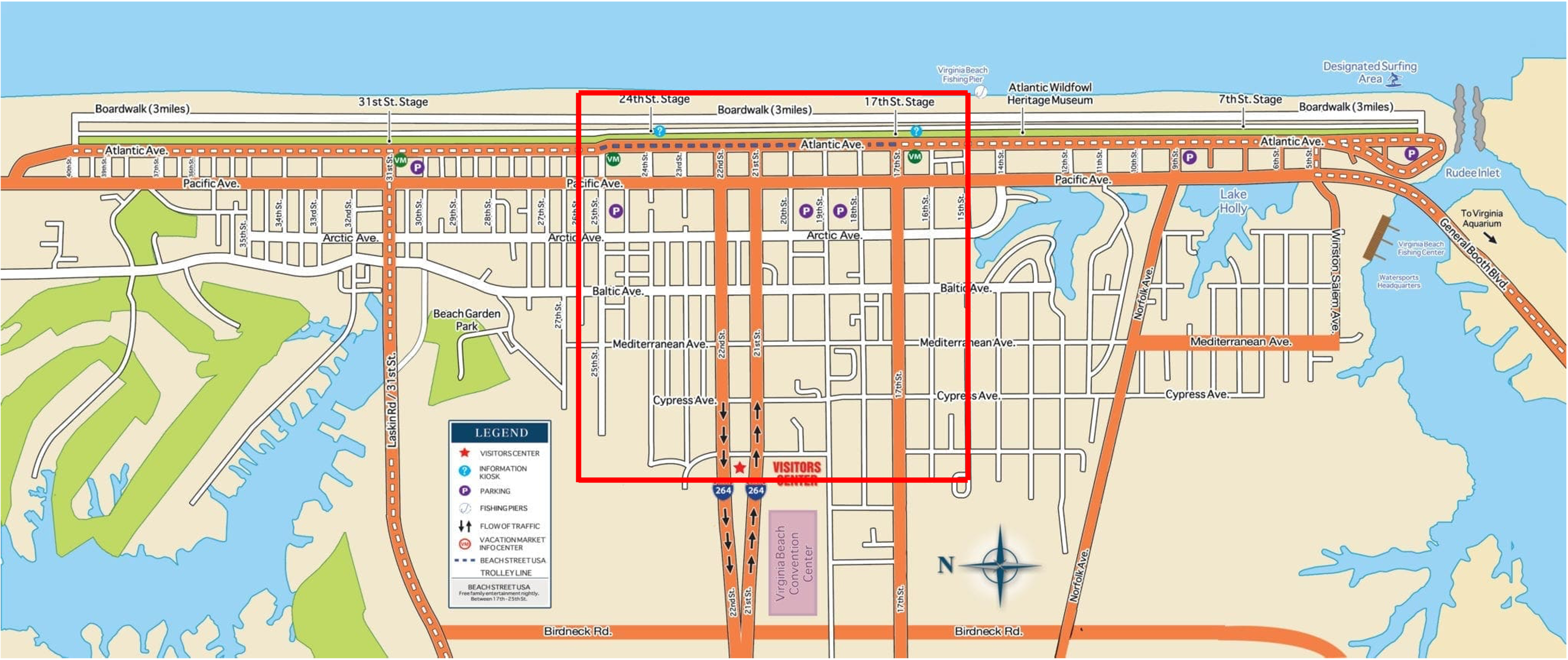}
\caption{ Case Study Area (Source: https://virginiabeach.guide/virtual-visitors-guide) }
\label{case1}
\end{figure}

\subsection{Case Study Setup}

\subsubsection{Spatial Distribution of OHCAs} \label{sohca}

We obtained historical OHCA data from the ``Virginia Beach OHCA Response Data Set” (VBOHCARDS) \citep{custodio2022spatiotemporal}. The VBOHCARDS recorded all the reported OHCA locations from January 1, 2017 to June 30, 2019. A total of 46 OHCA incidents that happened in the case study area were collected. Figure \ref{OHCA} illustrates the distribution of these incidents, indicating a concentration along the oceanfront and boardwalk. This data set serves as a valuable resource for our analysis of OHCA occurrences in the specified area.

\begin{figure}[htbp]
\centering
\includegraphics[width=1\textwidth]{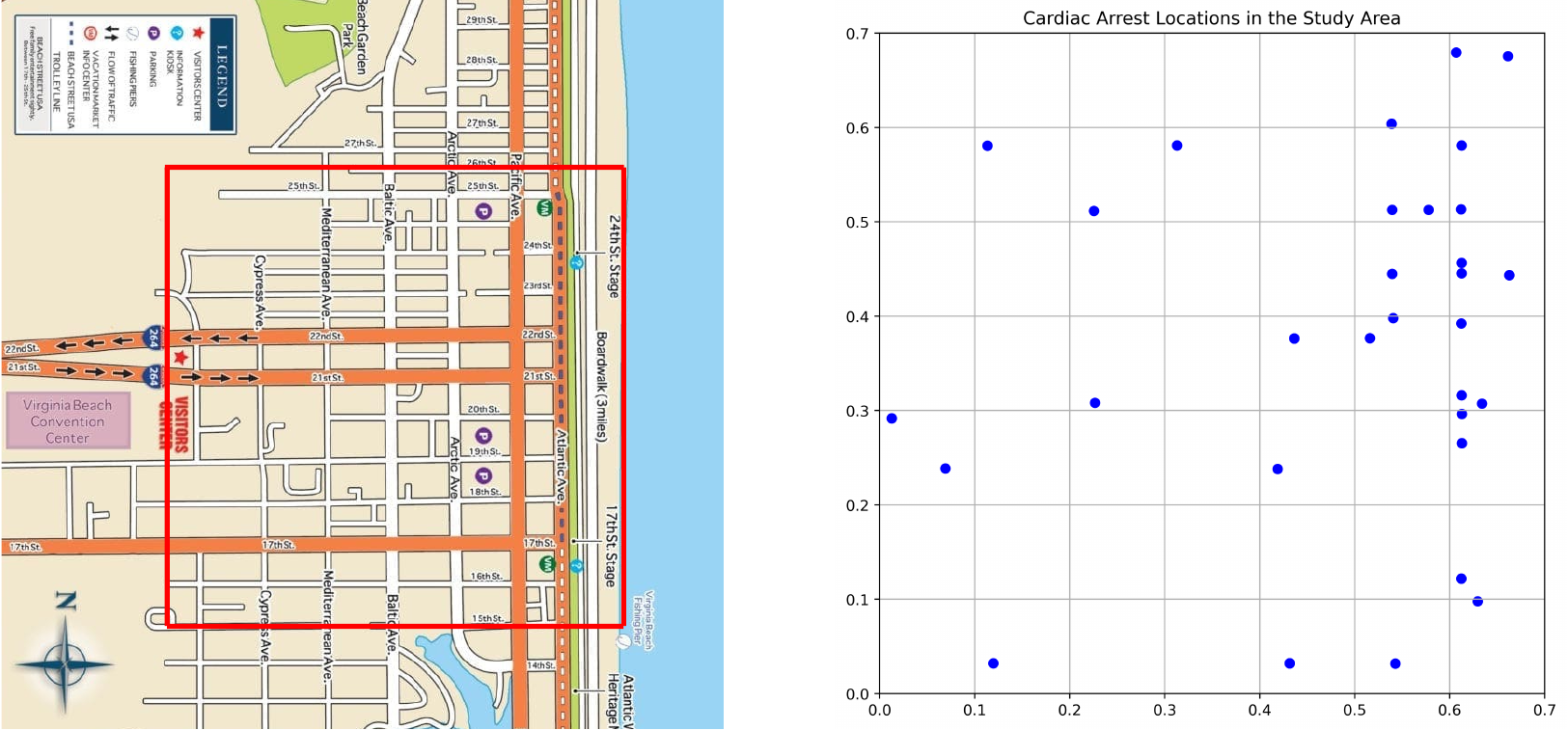}
\caption{Historical OHCA Data in the Case Study Area }
\label{OHCA}
\end{figure}

As shown in Figure \ref{KDE_1}, the spatial distribution of OHCA is estimated using kernel density estimation (KDE), following the procedure discussed in \citet{chan2016optimizing}. Additionally, a grid version of the KDE map is constructed (see Figure \ref{KDE_2}). The weight of each cell is calculated by averaging the density values within it, representing the probability of an OHCA incident occurring at the cell's center. It is worth noting that the difference in the probability distribution is solely reflected in the weight and does not affect the computational complexity. Consequently, we could either sample directly from the spatial distribution by the Monte Carlo method or utilize the grid weight map and quadrature for numerical integration.


\begin{figure}[htbp]
    \centering
    \begin{subfigure}{0.45\textwidth}
        \centering
        \includegraphics[width=\linewidth]{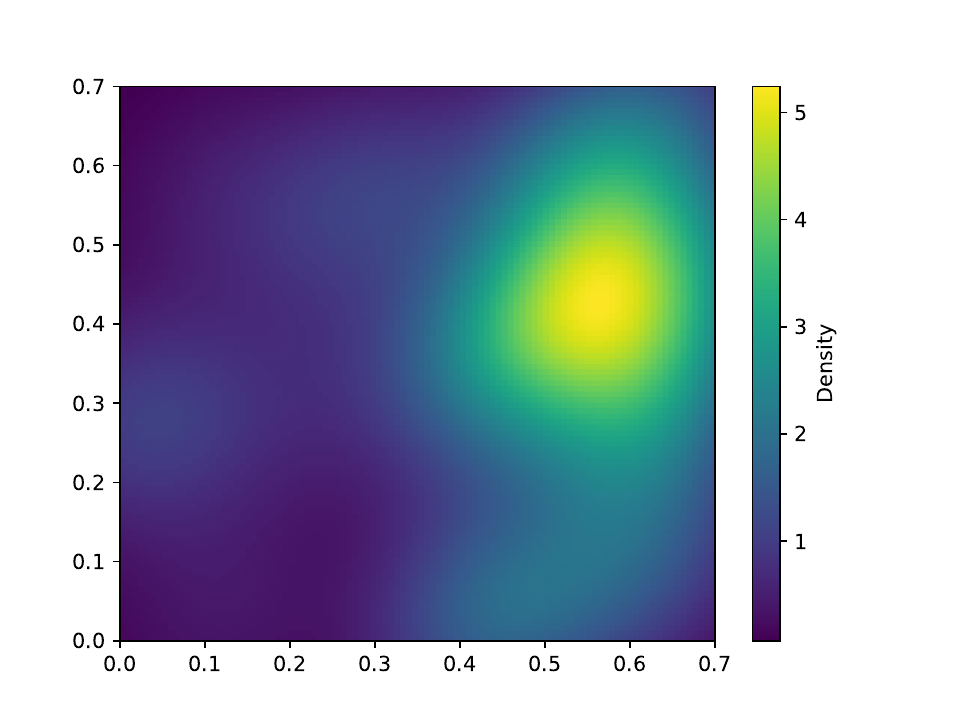}
        \caption{KDE Map}
        \label{KDE_1}
    \end{subfigure}
    \hspace{1em}
    \begin{subfigure}{0.45\textwidth}
        \centering
        \includegraphics[width=\linewidth]{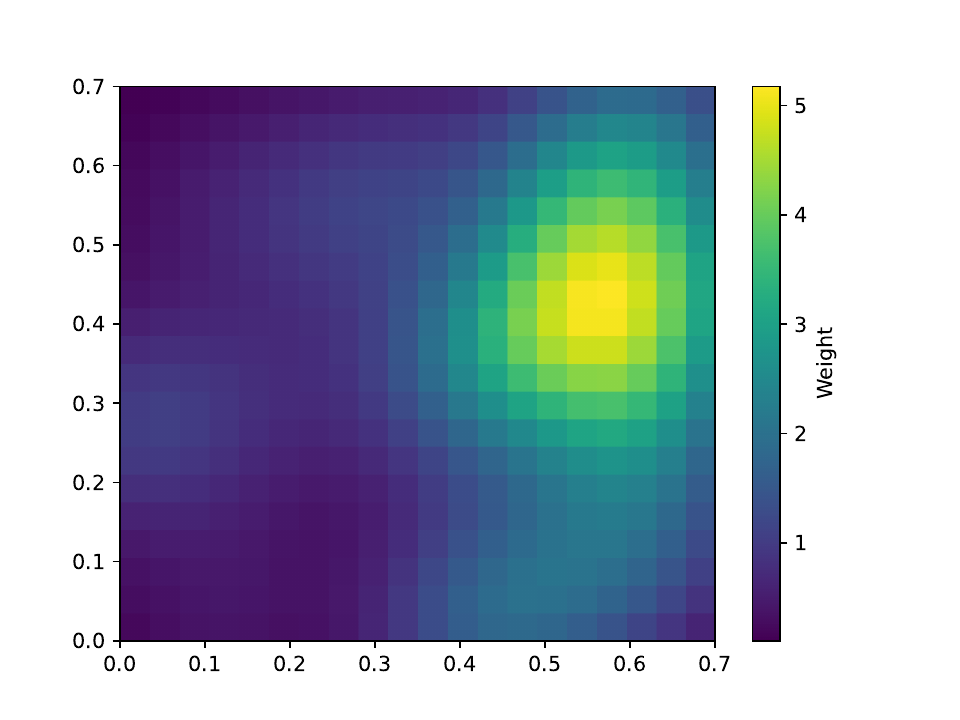}
        \caption{$20\times20$ Grid Weight Map}
        \label{KDE_2}
    \end{subfigure}
    \caption{Illustration of KDE and Grid Weight Maps}
    \label{KDE}
\end{figure}

\subsubsection{ Spatial Distribution of Volunteers}
\label{spatial}

Constructing the spatial distribution of volunteers presents significant challenges because of the dynamic and unpredictable nature of human movements. In this case study, we approximately characterize the spatial distribution of volunteers by dividing the study area into four zones. As shown in Figure \ref{Zone}, these zones are the Old Beach District (Z1), Arterial Traffic Zone (Z2), ViBe Creative District (Z3), and Virginia Beach Oceanfront (Z4), which match the existing district geography and layout.  

Assuming the distribution of a volunteer is uniform throughout a zone, we can estimate the likelihood of being in a zone by considering daily traffic volumes within zones. Specifically, we use Annual Average Daily Traffic (AADT) data \citep{AADT} to obtain a rough estimate of the population distribution. As depicted in Figure \ref{Zone}, the blue dots indicate the locations of AADT sensors, and the number in each zone corresponds to its average AADT value. By normalizing these values, we can obtain the probability of the volunteer appearing in each zone. Then, we construct a $20\times20$ grid weight map in the same way as in Section \ref{sohca}.

\begin{figure}[htbp]
\centering
\includegraphics[width=0.4\textwidth]{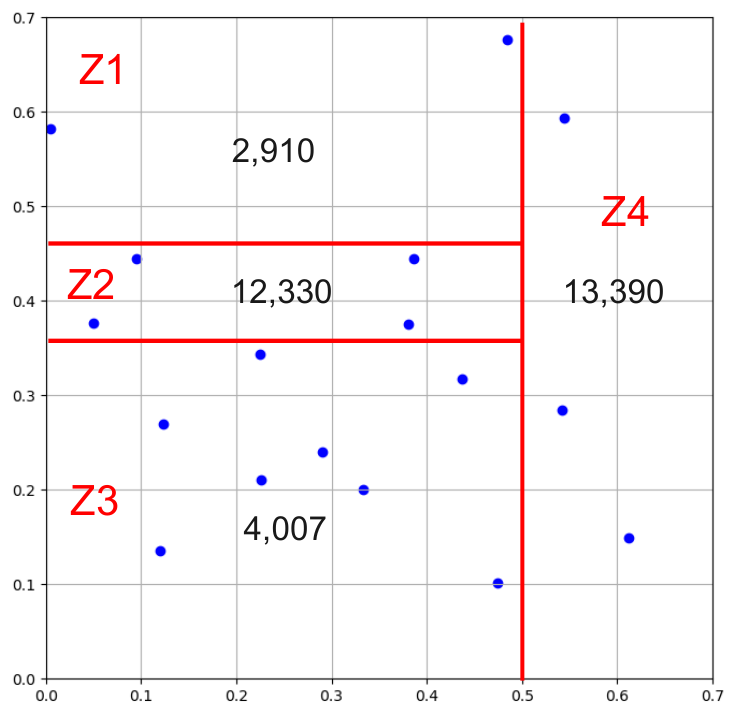}
\caption{Zones and Their Average AADT Values}
\label{Zone}
\end{figure}

\subsubsection{Experimental Design}

We conducted experiments to investigate the impact of the number of AEDs and the number of volunteers on the mean delivery distance (i.e., total travel distance). For numerical integration in the continuous volunteer and victim spaces, we utilize quadrature based on grid weight maps for the single-volunteer case and Monte Carlo integration with 3,000 samples for multiple-volunteer cases. Moreover, we use a $20\times20$ grid for the continuous solution space, where each cell's center is a potential hub location.  For the SPSA step, we choose 10,000 as the sample size for each iteration, while adhering to the settings outlined in Section \ref{SPSAsetting}. Additionally, we construct the mean delivery distance contour map based on the mean distance value of 1,000 samples at each location to visualize the spatial distribution of delivery distances.

\subsection{Computational Results}

\subsubsection{Numbers of AEDs and Volunteers}

Figure \ref{Obj_compare:1} illustrates the impact of varying the numbers of AEDs and volunteers on the mean delivery distance. Across all volunteer scenarios (i.e., 1, 5, 10, and 15 volunteers), the mean delivery distance decreases as the number of AEDs increases. This indicates that more AEDs reduce the distance volunteers need to travel to deliver them, which is expected. However, the reduction in the mean delivery distance becomes less pronounced as the number of AEDs increases, especially beyond five AEDs. This effect suggests diminishing returns with the addition of more AEDs. This decreasing marginal benefit trend also occurs as the number of volunteers increases. For a single volunteer, the mean delivery distance starts at around 0.58 miles with one AED and gradually decreases to just above 0.44 miles with nine AEDs. When more volunteers are involved, the mean delivery distance is significantly reduced. For instance, with five volunteers, the distance starts at around 0.41 miles and decreases to around 0.25 miles when the number of AEDs reaches nine. For the scenario with five AEDs, as shown in Figure \ref{Obj_compare:2}, diminishing returns with respect to the number of volunteers can also be observed. After reaching five volunteers, the rate of decrease in delivery distance slows down. When the number of volunteers increases from 10 to 40, the reduction in the mean delivery distance becomes marginal. 


\begin{figure}[htbp]
    \centering
    \begin{subfigure}{0.45\textwidth}
        \centering
        \includegraphics[width=\linewidth]{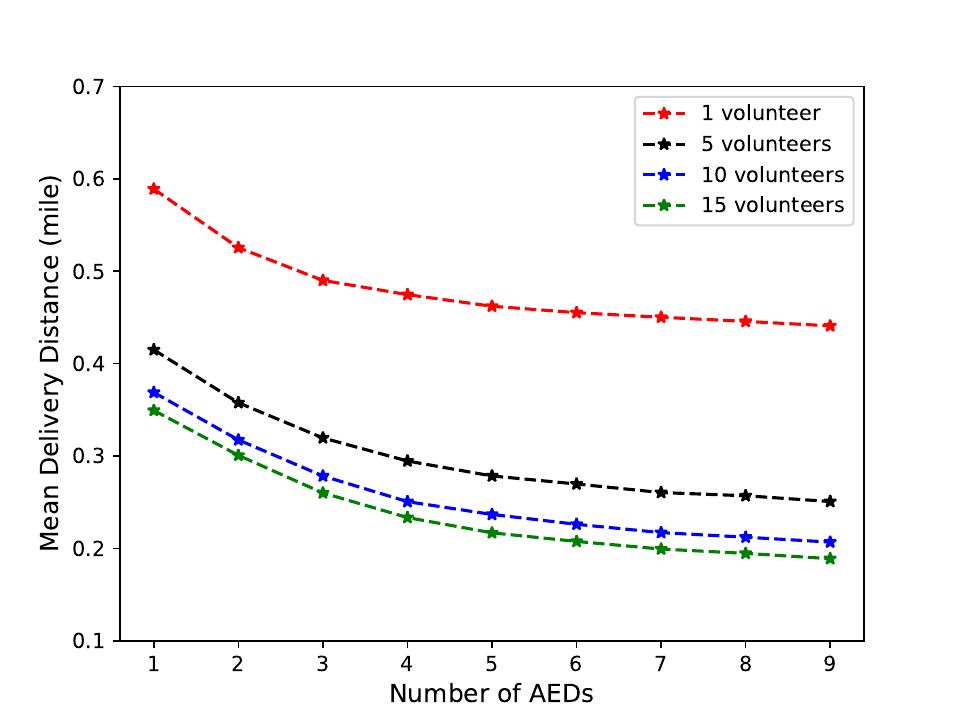}
        \caption{ }
        \label{Obj_compare:1}
    \end{subfigure}
    \hspace{1em}
    \begin{subfigure}{0.45\textwidth}
        \centering
        \includegraphics[width=\linewidth]{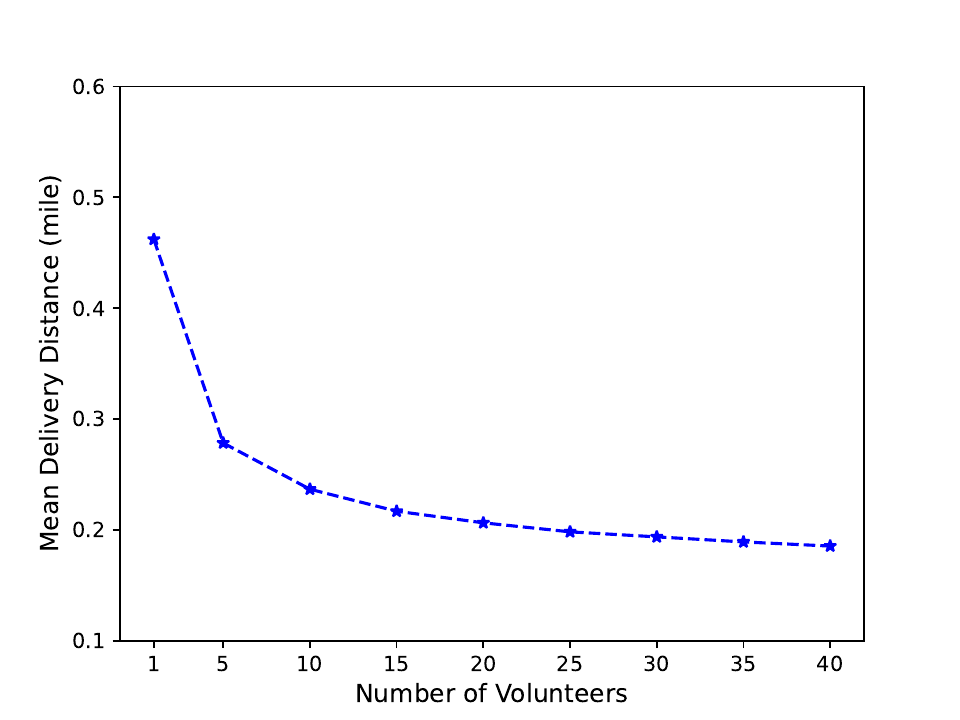}
        \caption{ }
        \label{Obj_compare:2}
    \end{subfigure}
    \caption{Mean Delivery Distance Comparison: (a) Impact of the Number of AEDs and (b) Impact of the Number of Volunteers with Five AEDs}
    \label{Obj_compare}
\end{figure}

\subsubsection{Mean Delivery Distance Contour Map}

Figure \ref{contour} illustrates the mean delivery distance contour map with five AEDs and ten volunteers. The red stars mark the locations of the five AEDs generated by our approximation method, while the color scale on the side represents the range of the mean delivery distance. The AED placements effectively cover those regions with a high number of OHCA incidents and population density. Additionally, the contour map reveals concentric diamond patterns around each AED, with the mean delivery distance increasing as one moves away from an AED. These diamond patterns arise from our utilization of the Manhattan metric.

\begin{figure}[htbp]
\centering
\includegraphics[width=0.5\textwidth]{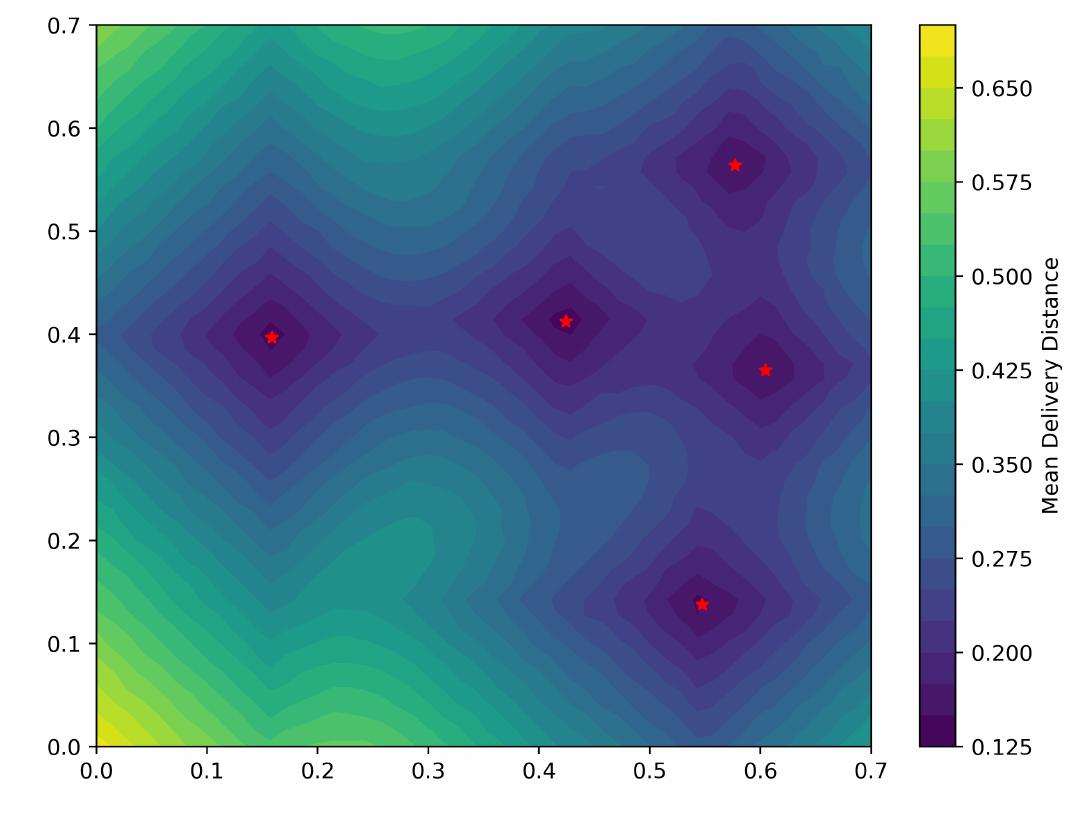}
\caption{ Mean Delivery Distance Contour Map for Five AEDs and Ten Volunteers}
\label{contour}
\end{figure}

\subsection{Sensitivity Analysis}

To evaluate the robustness of our solutions, we conduct a sensitivity analysis by varying the number of volunteers and their spatial distribution.

\subsubsection{Number of Volunteers}

Figure \ref{SA_Volunteer} shows the resulting AED locations with different numbers of volunteers. The relatively tight cluster of points within each group indicates the solutions are stable for the multi-volunteer cases. One exception outside the cluster is the single-volunteer case, marked as black. The resulting outlier phenomenon is consistent with the impact result shown in Figure \ref{Obj_compare:1}.
\begin{figure}[htbp]
\centering
\includegraphics[width=0.45\textwidth]{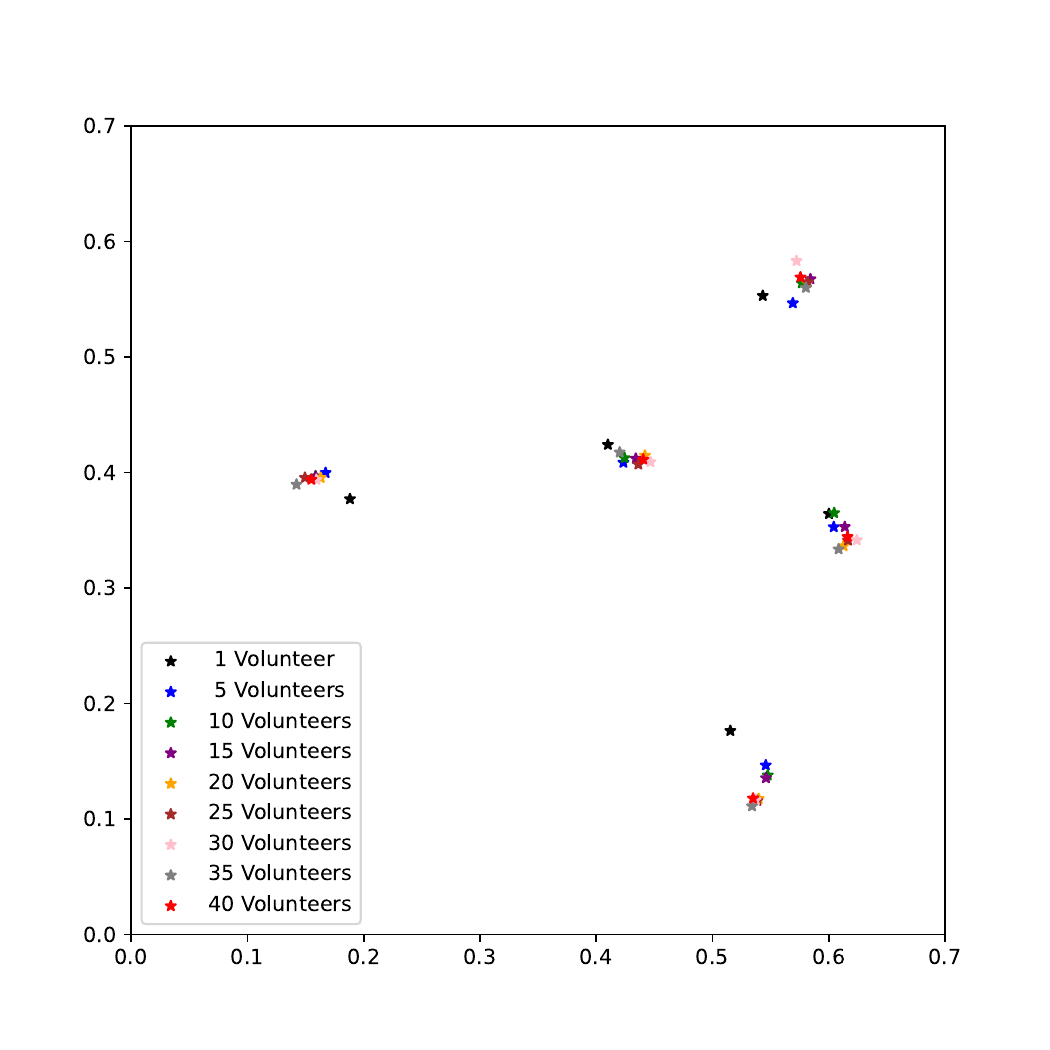}
\caption{Sensitivity Analysis of the Number of Volunteers with Five AEDs}
\label{SA_Volunteer}
\end{figure}

\subsubsection{Volunteer Distributions}

We obtained the average AADT values $[2,910, 12,330, 4,007, 13,390]$ for Zones Z1, Z2, Z3, and Z4, respectively, to estimate each zone's population in Section \ref{spatial}. Here, we use these values as a benchmark and adjust them to determine how they affect the obtained solutions. Specifically, each zone's value is adjusted up or down by some percentage, and the value change is proportionately distributed to the other three zones according to their benchmark values. We test 10\% and 50\% changes, and generate eight scenarios for each, as detailed in Table \ref{scenarios}.

\begin{table}[h!]
\centering
\caption{Sensitivity Analysis Scenarios for Volunteer Distributions}
\label{scenarios}
\begin{tabular}{@{}lcc@{}}
\toprule
Scenario & 10\% Change                     & 50\% Change                     \\ 
\midrule
1        & [3,201, 12,209, 3,968, 13,259] & [4,365, 11,727, 3,811, 12,735] \\ 
2        & [2,619, 12,451, 4,046, 13,521] & [1,455, 12,933, 4,203, 14,045] \\ 
3        & [2,733, 13,563, 3,764, 12,577] & [2,027, 18,495, 2,791, 9,325]  \\ 
4        & [3,087, 11,097, 4,250, 14,203] & [3,793, 6,165, 5,223, 17,455]  \\ 
5        & [2,869, 12,157, 4,408, 13,203] & [2,706, 11,467, 6,010, 12,453] \\ 
6        & [2,951, 12,503, 3,606, 13,577] & [3,114, 13,193, 2,004, 14,327] \\ 
7        & [2,708, 11,472, 3,728, 14,729] & [1,898, 8,041, 2,613, 20,085]  \\ 
8        & [3,112, 13,188, 4,286, 12,051] & [3,922, 16,619, 5,401, 6,695]  \\ 
\bottomrule
\end{tabular}
\end{table}


The resulting AED locations for the different scenarios with five AEDs and ten volunteers are shown in Figure \ref{SA_distribution}. The 10\% population adjustment leads to minimal shifts in the positioning of AEDs across different scenarios (see Figure \ref{SA_distribution:1}). All scenarios are closely clustered around the benchmark locations, indicating that minor adjustments have little effect on the solutions. The 50\% adjustment introduces more deviations from the benchmark AED locations, although they remain clustered within certain regions (see Figure \ref{SA_distribution:2}). This result shows that our solutions are stable and robust to perturbations in the spatial distribution of volunteers, even with somewhat large adjustments. 

The observed robustness can be attributed to the unchanged population density ranks across zones as population volumes change. Specifically, Zones Z2 and Z4 consistently maintain the top two population volumes, meaning these areas continue to be prioritized in the optimization process, regardless of the adjustment scale. As a result, despite population volume changes, the underlying population distribution pattern remains unchanged, leading to minimal impact on AED location decisions and ensuring reliable performance across various scenarios.


\begin{figure}[htbp]
    \centering
    \begin{subfigure}{0.45\textwidth}
        \centering
        \includegraphics[width=\linewidth]{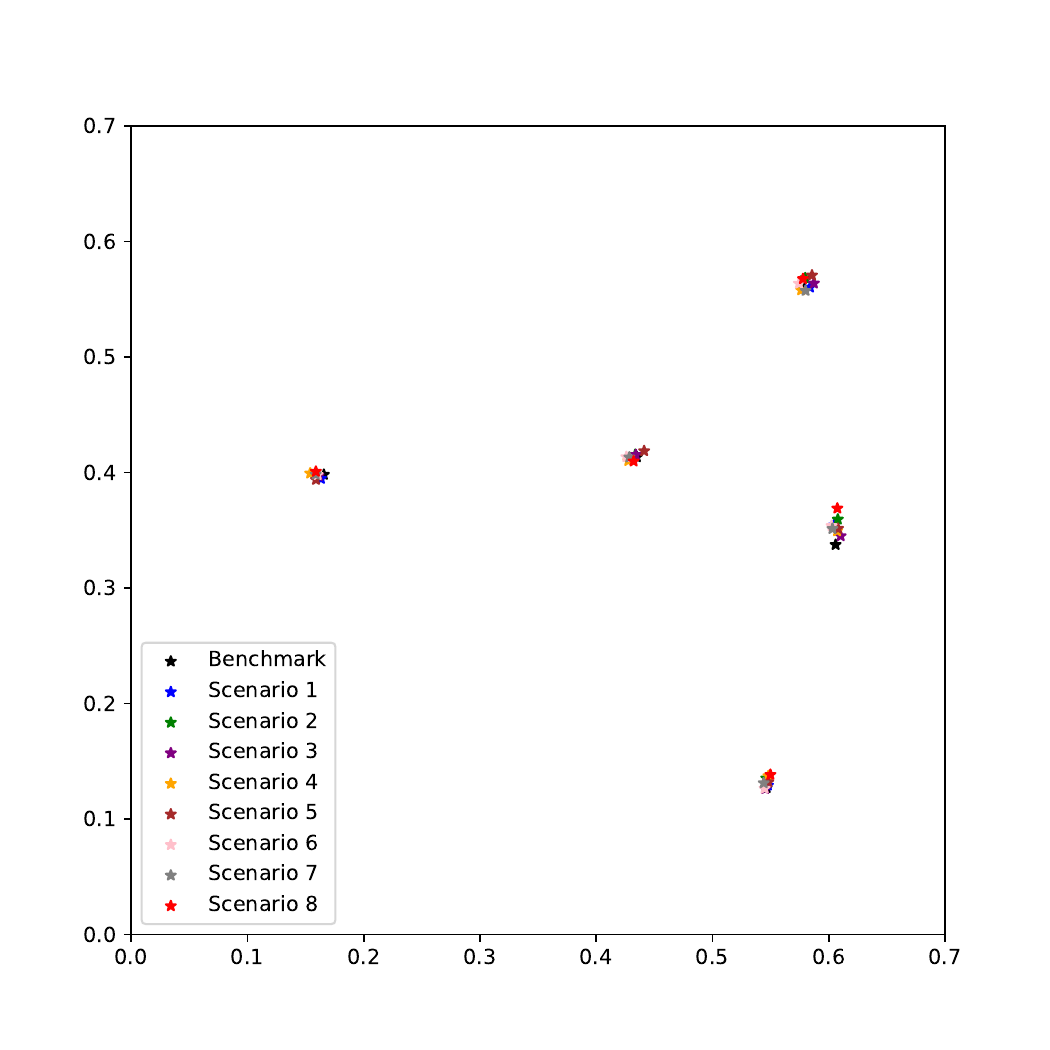}
        \caption{AED Locations for 10\% Adjustment}
        \label{SA_distribution:1}
    \end{subfigure}
    \hspace{1em}
    \begin{subfigure}{0.45\textwidth}
        \centering
        \includegraphics[width=\linewidth]{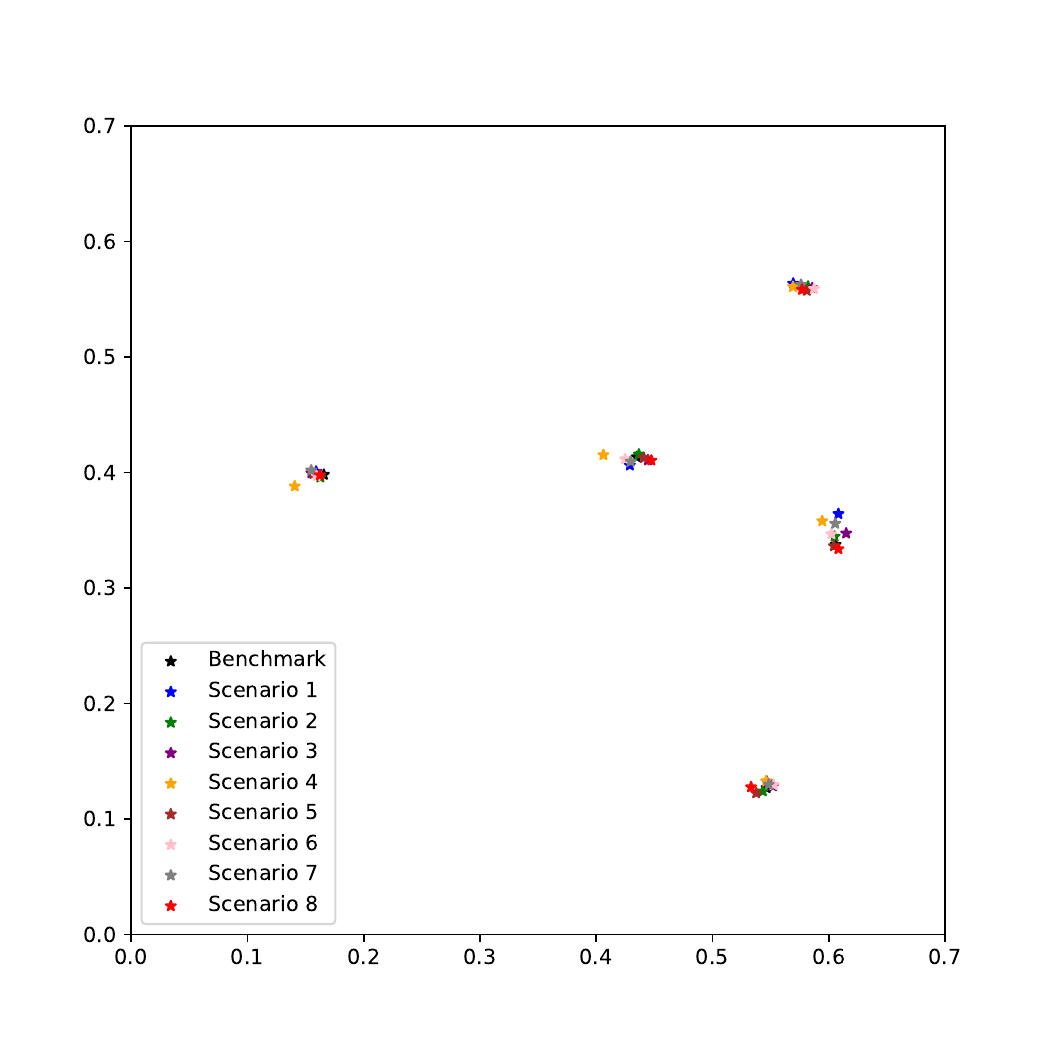}
        \caption{AED Locations for 50\% Adjustment}
        \label{SA_distribution:2}
    \end{subfigure}
    \caption{Sensitivity Analysis of Volunteer Distributions with Five AEDs}
    \label{SA_distribution}
\end{figure}

\subsection{Discussion}

Based on the computational results, the importance of volunteers should be emphasized. While increasing the number of AEDs is beneficial, the most significant improvements in delivery distance are associated with an increase in the number of volunteers. Our observations reveal that even with just one AED installed, having five volunteers results in a smaller mean delivery distance compared to nine AEDs with only one volunteer. This insight suggests that recruiting more volunteers can have a greater impact than solely increasing AEDs, especially in areas with limited volunteer availability. Increasing the number of volunteers can be achieved through enhanced community outreach and education. Furthermore, the diminishing returns observed with an increase in the number of AEDs or volunteers indicate that an optimal combination of these resources may exist, taking into account the associated socioeconomic costs. In practice, there may be ideal thresholds beyond which additional effort may not substantially improve response efficiency. The sensitivity analysis results demonstrate that our approximation method is robust, making it suitable for algorithm deployment. By estimating the number of volunteers roughly instead of providing an exact number, we can effectively achieve our objective. When it comes to the spatial distribution of volunteers, dividing the area appropriately and providing the population density ranks suffice, even in the absence of accurate population census data. Such stability yields substantial practical benefits.

\section{Conclusions and Future Research}\label{sec:Conclusion}

In this paper, we revisited continuous $p$-hub location problems under the $\ell^1$ metric, where a random service provider will serve a random customer via one hub. For the one-dimensional case, an analytical expression for the objective function is derived by an innovative integration method. The core basis of this method is in determining the integration ranges within which a specific hub will be utilized. Based on the analytical expression, we demonstrate the convexity of the problem and obtain a closed-form solution for the optimal locations. We extend the integration method to the two-dimensional cases, and obtain analytical expressions for up to two hubs. The exponential growth in the number of integration area divisions makes the direct integration method challenging for the general case of $n$ hubs. Therefore, we propose a two-step approximation method based on simulation optimization. After discretization, our problem can be transformed into a large-scale $p$-median problem, which can be solved efficiently using Benders decomposition. SPSA is then utilized to improve the solution quality further in the continuous solution space. We also demonstrate that the approximation method can be applied to the problem with multiple providers. We implement our methodology for deploying public-access AEDs in Virginia Beach. The findings suggest that, while increasing the number of AEDs is beneficial, we should emphasize public education to recruit a sufficient number of volunteers.

We propose two potential directions for future work. Firstly, although our approximation method can be easily extended to other metrics, our closed-form derivations are only applicable to the $\ell^1$ metric. Exploring the shapes of integration regions and properties under the $\ell^2$ metric represents an important direction for future investigation. Research on the $\ell^2$ metric provides an alternate perspective, encompassing the Euclidean distance and offering a more common measure of real-world travel.  Additionally, our study focuses on squares or rectangles, representing regularly shaped regions without internal barriers. While this offers a clean, simplified environment for analysis, it may not be fully representative of practical applications, where obstacles such as buildings, parks, or restricted zones often prohibit building facilities. Future work could incorporate barriers within the study area, providing critical insights into how barriers influence optimal hub placement. Addressing these two directions would enhance the practical applicability of our models and methodologies. 

\bibliographystyle{trb}
\bibliography{references.bib}

\newpage

\section{Appendix}
\setcounter{equation}{0}
\setcounter{figure}{0}
\renewcommand{\theequation}{A.\arabic{equation}}
\renewcommand{\thefigure}{A.\arabic{figure}}

\subsection{Proofs of Propositions and Corollaries}

\label{Appendix:Proposition}

\begin{proof}{\bf{Proof of Proposition \ref{p1}}.}

For any $k \in [n-1]$ and $i \in [n-k]$, we want to prove that if $0 \le x \le p_k$ and $0 \le y \le 1$, we have $|x-p_k|+|y-p_k| \le |x-p_{k+i}|+|y-p_{k+i}|,$ that is, $|y-p_k|-|y-p_{k+i}| \le |x-p_{k+i}|-|x-p_k|.$ Moreover, if $0 \le x \le p_k,$ we have $|x-p_{k+i}|-|x-p_k| = p_{k+i} - x - (p_k - x) = p_{k+i} - p_{k}$. So, we want to prove $|y-p_k|-|y-p_{k+i}| \le p_{k+i} - p_{k}.$

We have the following three cases:

\noindent Case 1. $0 \le x \le p_k$ and $0 \le y < p_{k}.$

We have $|y-p_k|-|y-p_{k+i}| = p_k - y - (p_{k+i} - y) = p_k - p_{k+i}$. Therefore,  $|y-p_k|-|y-p_{k+i}| - (p_{k+i} - p_{k}) =  
p_k - p_{k+i} - p_{k+i} + p_k = 2(p_k - p_{k+i}) < 0$.

\noindent Case 2. $0 \le x \le p_k$ and $p_k \le y < p_{k+i}.$

We have $|y-p_k|-|y-p_{k+i}| = y - p_k - (p_{k+i} - y) = 2y - p_k - p_{k+i}$. Therefore,  $|y-p_k|-|y-p_{k+i}| - (p_{k+i} - p_{k}) = 2y - p_k - p_{k+i} - p_{k+i} + p_k = 2(y - p_{k+i}) < 0$.

\noindent Case 3. $0 \le x \le p_k$ and $p_{k+i} \le y \le 1.$

We have $|y-p_k|-|y-p_{k+i}| = y-p_k - (y-p_{k+i}) = p_{k+i} -p_k $. Therefore,  $|y-p_k|-|y-p_{k+i}| - (p_{k+i} - p_{k}) = p_{k+i} -p_k - p_{k+i} + p_k = 0$.  \end{proof}

\begin{proof}{\bf{Proof of Corollary \ref{cor1}}.}

 By setting $k=1$ in Proposition \ref{p1}, we have $d_{1} \le d_{1+i}$ for all $i \in [n-1]$, that is, $d_1 \leq d_2, d_1 \leq d_3, \ldots, d_1 \leq d_n.$ Therefore,  $d_1  = \min_{i \in [n]} d_{i}$.  \end{proof}

\begin{proof}{\bf{Proof of Proposition \ref{p2}}.}
 
For any $k \in \{2,\ldots,n\}$ and $i \in [k-1]$, we want to prove that if $p_k \le x \le 1$ and $0 \le y \le 1$, we have $|x-p_k|+|y-p_k| \le |x-p_{k-i}|+|y-p_{k-i}|,$ that is, $|y-p_k|-|y-p_{k-i}| \le |x-p_{k-i}|-|x-p_k|.$ Moreover, if $p_k \le x \le 1,$ we have $|x-p_{k-i}|-|x-p_k| = x - p_{k-i} - (x - p_k) = p_k - p_{k-i}$. So, we want to prove $|y-p_k|-|y-p_{k-i}| \le p_k - p_{k-i}.$

We have the following three cases:

\noindent Case 1. $p_k \le x \le 1$ and $0 \le y < p_{k-i}.$

We have $|y-p_k|-|y-p_{k-i}| = p_k - y - (p_{k-i} - y) = p_k - p_{k-i}$. Therefore,  $|y-p_k|-|y-p_{k-i}| - (p_k - p_{k-i}) = p_k - p_{k-i} - p_k + p_{k-i} = 0$.

\noindent Case 2. $p_k \le x \le 1$ and $p_{k-i} \le y < p_k.$

We have $|y-p_k|-|y-p_{k-i}| =  p_k - y - (y - p_{k-i}) = p_k + p_{k-i} - 2y$. Therefore,  $|y-p_k|-|y-p_{k-i}| - (p_k - p_{k-i}) = p_k + p_{k-i} - 2y - p_k + p_{k-i} = 2(p_{k-i}-y) \le 0$.

\noindent Case 3. $p_k \le x \le 1$ and $p_k \le y \le 1.$

We have $|y-p_k|-|y-p_{k-i}| = y-p_k - (y-p_{k-i}) = p_{k-i} -p_k $. Therefore,  $|y-p_k|-|y-p_{k-i}| - (p_k - p_{k-i}) = p_{k-i} -p_k - p_k + p_{k-i} = 2(p_{k-i} - p_k) < 0$.   \end{proof}

\begin{proof}{\bf{Proof of Corollary \ref{cor2}}.}

By setting $k=n$ in Proposition \ref{p2}, we have $d_{n} \le d_{n-i}$ for all $i \in [n-1]$, that is, $d_n \leq d_{n-1}, d_n \leq d_{n-2}, \ldots, d_n \leq d_1.$ Therefore,  $d_n  = \min_{i \in [n]} d_{i}$.   \end{proof}

\begin{proof}{\bf{Proof of Proposition \ref{p3}}.}

For any $k \in [n-1]$, we want to prove that if $p_k \le x \le p_{k+1}$ and $0 \le y \le p_k + p_{k+1} -x,$ we have $d_k \le d_1, \ldots, d_k \le d_n.$ Since $p_k \le x \le p_{k+1}$, based on Proposition \ref{p1} and Proposition \ref{p2}, it suffices to prove that $d_k \le d_{k+1}$, that is, $|x-p_k|+|y-p_k| \le |x-p_{k+1}|+|y-p_{k+1}|$. The inequality is equivalent to $|y-p_k|-|y-p_{k+1}| \le |x-p_{k+1}|-|x-p_k|.$ Moreover, if $p_k \le x \le p_{k+1},$ we have $|x-p_{k+1}|-|x-p_k| = p_{k+1} - x - (x - p_k) = p_{k+1} + p_k -2x$. So, we want to prove $|y-p_k|-|y-p_{k+1}| \le p_{k+1} + p_k -2x.$

We have the following two cases:

\noindent Case 1. $p_k \le x \le p_{k+1}$ and $0 \le y < p_k.$

We have $|y-p_k|-|y-p_{k+1}| = p_k - y - (p_{k+1} - y) = p_k - p_{k+1}$. Therefore,  $|y-p_k|-|y-p_{k+1}| - (p_{k+1} + p_k -2x) =  
p_k - p_{k+1} - p_{k+1} - p_k + 2x = 2(x-p_{k+1}) \le 0$.

\noindent Case 2. $p_k \le x \le p_{k+1}$ and $p_k \le y \le p_k + p_{k+1} - x.$

We have $|y-p_k|-|y-p_{k+1}| = y - p_k - (p_{k+1} - y) = 2y - p_k - p_{k+1}$. Therefore, $|y-p_k|-|y-p_{k+1}| - (p_{k+1} + p_k -2x) =  
2y - p_k - p_{k+1} - p_{k+1} - p_k + 2x = 2(x+y-p_k-p_{k+1}) \le 0$.  \end{proof}

\begin{proof}{\bf{Proof of Proposition \ref{p4}}.}

For any $k \in [n-1]$, we want to prove that if $p_k \le x \le p_{k+1}$ and $p_k + p_{k+1} - x \le y \le 1,$ we have $d_{k+1} \le d_1, \ldots, d_{k+1} \le d_n.$ Since $p_k \le x \le p_{k+1}$, based on Proposition \ref{p1} and Proposition \ref{p2}, it suffices to prove that $d_k \ge d_{k+1}$, that is, $|x-p_k|+|y-p_k| \ge |x-p_{k+1}|+|y-p_{k+1}|$. The inequality is equivalent to $|y-p_k|-|y-p_{k+1}| \ge |x-p_{k+1}|-|x-p_k|= p_{k+1} + p_k -2x.$

We have the following two cases:

\noindent Case 1. $p_k \le x \le p_{k+1}$ and $p_k + p_{k+1} - x \le y < p_{k+1}.$

We have $|y-p_k|-|y-p_{k+1}| = y - p_k - (p_{k+1} - y) = 2y - p_k - p_{k+1}$. Therefore, $|y-p_k|-|y-p_{k+1}| - (p_{k+1} + p_k -2x) =  
2y - p_k - p_{k+1}- p_{k+1} - p_k + 2x = 2(x + y - p_k - p_{k+1}) \ge 0$.

\noindent Case 2. $p_k \le x \le p_{k+1}$ and $p_{k+1} \le y \le 1.$

We have $|y-p_k|-|y-p_{k+1}| = y - p_k - (y -p_{k+1}) = p_{k+1} - p_k$. Therefore, $|y-p_k|-|y-p_{k+1}| - (p_{k+1} + p_k - 2x) =  
p_{k+1} - p_k - p_{k+1} - p_k + 2x = 2(x-p_k) \ge 0$.   \end{proof}

\begin{proof}{\bf{Proof of Proposition \ref{2Dp1}}.}

It is equivalent to show that when $0 \leq x_1 \leq p_{11}$, $0 \leq x_2 \leq p_{12}$, $0 \leq y_1 \leq 1$, and $ 0 \leq y_2 \leq 1$, we have $d_1 \le d_2$, that is,
\begin{align}
\label{2d1}
&|x_1-p_{11}|+|x_2-p_{12}|+|y_1-p_{11}|+|y_2-p_{12}| \nonumber
\\& \le |x_1-p_{21}|+|x_2-p_{22}|+|y_1-p_{21}|+|y_2-p_{22}|.
\end{align}
Based on Corollary \ref{cor1} with $n=2$, if $0 \leq x_1 \leq p_{11}$ and $0 \leq y_1 \leq 1$, we have 
\begin{equation}
\label{2d1A}
|x_1-p_{11}|+|y_1-p_{11}| \le |x_1-p_{21}|+|y_1-p_{21}|.
\end{equation}
Similarly, if $0 \leq x_2 \leq p_{12}$ and $0 \leq y_2 \leq 1$, we have
\begin{equation}
\label{2d1B}
|x_2-p_{12}|+|y_2-p_{12}| \le |x_2-p_{22}|+|y_2-p_{22}|.
\end{equation}
Combining Inequality \ref{2d1A} and Inequality \ref{2d1B}, Inequality \ref{2d1} holds. \end{proof}

\begin{proof}{\bf{Proof of Proposition \ref{2Dp2}}.}

It is equivalent to show that when $p_{21} \leq x_1 \leq 1$, $p_{22} \leq x_2 \leq 1$, $0 \leq y_1 \leq 1$, and $ 0 \leq y_2 \leq 1$, we have $d_2 \le d_1$, that is,
\begin{align}
\label{2d2}
&|x_1-p_{21}|+|x_2-p_{22}|+|y_1-p_{21}|+|y_2-p_{22}| \nonumber
\\& \le |x_1-p_{11}|+|x_2-p_{12}|+|y_1-p_{11}|+|y_2-p_{12}|.
\end{align}
Based on Corollary \ref{cor2} with $n=2$, if $p_{21} \leq x_1 \leq 1$ and $0 \leq y_1 \leq 1$, we have 
\begin{equation}
\label{2d2A}
|x_1-p_{21}|+|y_1-p_{21}| \le |x_1-p_{11}|+|y_1-p_{11}|.
\end{equation}
Similarly, if $p_{22} \leq x_2 \leq 1$ and $0 \leq y_2 \leq 1$, we have
\begin{equation}
\label{2d2B}
|x_2-p_{22}|+|y_2-p_{22}| \le |x_2-p_{12}|+|y_2-p_{12}|.
\end{equation}
Combining Inequality \ref{2d2A} and Inequality \ref{2d2B}, Inequality \ref{2d2} holds.  \end{proof}

\subsection{Calculation of $\Delta F_j$ in the One-Dimensional Case}
\label{Appendix:F_j}
Since
\begin{align*}
I_{j-1} &= \int^{p_{j}}_{p_{j-1}} \int^{p_{j-1}+p_{j}-x}_0 d_{j-1} dydx\\
&+\int^{p_{j}}_{p_{j-1}} \int_{p_{j-1}+p_{j}-x}^1 d_{j} \,dy \,dx\\ &=  \int^{p_{j}}_{p_{j-1}}  \left[\int^{p_{j-1}}_0  (x-y)dy  \right.\\ 
&\left.+\int^{p_{j-1}+p_{j}-x}_{p_{j-1}} \,(x+y-2p_{j-1}) \,dy\,\right] dx \\ 
&+\int^{p_{j}}_{p_{j-1}}  \left[\int^{p_j}_{p_{j-1}+p_{j}-x} \,(2 p_{j}-x -y)dy\right.\\
&\left. +\int^{1}_{p_j} \,(y-x) \,dy\,\right] dx \\& = -\frac{2}{3} p_{j-1}^{3} + \frac{2}{3} p_{j}^{3} + p_{j -1}^{2} p_{j} - p_{j-1} p_{j}^{2} \\
&+\frac{1}{2} p_{j-1}^{2} -\frac{1}{2} p_{j}^{2} - \frac{1}{2} p_{j-1}+\frac{1}{2} p_{j}, 
\end{align*}

\begin{align*}
\int_{p_{j}}^1 \int^1_0 d_j \,dy\,dx &= \int_{p_{j}}^1 \left[\int^{p_j}_0  \,(x-y)\,dy \right. \\ & \left. +\int^{1}_{p_j} \,(x+y-2p_j) \,dy\,\right]dx \\ &= -p_{j}^{3} + \frac{5}{2} p_{j}^{2}-\frac{5}{2} p_{j} +1,
\end{align*}

and
\begin{align*}
\int_{p_{j-1}}^1 \int^1_0 d_{j-1} \,dy \,dx &= \int_{p_{j}}^1 \left[\int^{p_{j-1}}_0 \,(x - y)dy \right. \\ & \left.+\int^{1}_{p_{j-1}} \,(x+y-2p_{j-1}) \,dy\,\right]dx \\ &= -p_{j -1}^{3} + \frac{5}{2} p_{j-1}^{2}-\frac{5}{2} p_{j-1}+1,
\end{align*}
we have
\begin{align*}
\Delta F_j  & =  I_{j-1}+\int_{p_{j}}^1 \int^1_0 d_j dydx\\
&-\int_{p_{j-1}}^1 \int^1_0 d_{j-1} \,dy \,dx  \\
 & = \frac{1}{3} p_{j -1}^{3} -\frac{1}{3} p_{j}^{3} + p_{j-1}^{2} p_{j} - p_{j-1} p_{j}^{2} \\
 & -2 p_{j-1}^{2} +2 p_{j}^{2} + 2 p_{j-1}-2 p_{j}. 
\end{align*}

\subsection{Derivation of $\mathbf{H}_n$}

\label{Appendix:H_n}

Since $\frac{\partial F_n}{\partial p_1}  = p_1^2 + 2 p_1 p_2 - p_2^2$, we have $$\frac{\partial^2 F_n}{\partial p_1^2} = 2p_1 + 2p_2 \mbox{ and } \frac{\partial^2 F_n}{\partial p_1 \partial p_2} = 2p_1 - 2p_2.$$
For $2 \le i \le n-1$, since $
\frac{\partial F_n}{\partial p_i} = p_{i-1}^2 - 2p_{i-1}p_i + 2p_i p_{i+1} - p_{i+1}^2,$ we have 
$$\frac{\partial^2 F_n}{\partial p_{i} p_{i-1}} =2p_{i-1} - 2p_i, \frac{\partial^2 F_n}{\partial p_i^2}= 2p_{i+1} -2p_{i-1}, \mbox{ and } \frac{\partial^2 F_n}{\partial p_i \partial p_{i+1}} = 2p_i - 2p_{i+1}.$$
Since $\frac{\partial F_n}{\partial p_n} =-p_n^2 + p_{n-1}^2 - 2p_{n-1} p_n + 4p_n - 2$, we have
$$\frac{\partial^2 F_n}{\partial p_n p_{n-1}} = 2p_{n-1} - 2p_n  \mbox{ and } \frac{\partial^2 F_n}{\partial p_n^2} =  4-2p_{n-1}-2p_n.$$
Other elements are equal to 0. 

\subsection{Closed-Form Optimal Solution for the One-Dimensional Case}

\label{Appendix:Closed One}

Since the optimization model for the one-dimensional case is convex, the optimal solution could be obtained through solving a system of $n$ equations, that is, $\nabla F_n = \mathbf{0}.$ 

For the first equation, based on Equation (\ref{F_n}) and Equation (\ref{Delta}), only $F_{1}$ and $\Delta F_{2}$ involve $p_{1}.$ Based on their expressions, we have that 
$$\frac{\partial F_n}{\partial p_1} = 4p_1 - 2 + p_1^2 + 2 p_1 p_2 - p_2^2 -4p_1 + 2 = p_1^2 + 2 p_1 p_2 - p_2^2 = 0,$$
which is equivalent to $(p_1 + p_2)^2 = 2p_2^2.$ Therefore, we have $p_2 = (\sqrt 2 + 1)p_1$.

For any $2 \le i \le n-1$, based on Equation (\ref{F_n}) and Equation (\ref{Delta}), only $\Delta F_{i}$ and $\Delta F_{i+1}$ involve $p_{i}.$ Based on their expressions, we have that 
\begin{align*}
\frac{\partial F_n}{\partial p_i} & = - p_i^2 + p_{i-1}^2 - 2p_{i-1}p_i + \\ &4 p_i -2 + p_i^2 + 2p_i p_{i+1} - p_{i+1}^2 - 4 p_i + 2  \\
& =  p_{i-1}^2 - 2p_{i-1}p_i + 2p_i p_{i+1} - p_{i+1}^2 = 0,
\end{align*}
which is equivalent to $(p_{i-1}-p_{i+1})(p_{i-1}-2p_i+p_{i+1})=0.$ Therefore, we have $p_{i+1} = 2p_{i} - p_{i-1}$. Thus, we obtain that $ p_{i+1}=(i\sqrt 2+1)p_1.$ 

For the last equation, only $\Delta F_{n}$ involve $p_{n}$ based on Equation (\ref{F_n}) and Equation (\ref{Delta}). Along with its expression, we have that
\begin{equation}
\frac{\partial F_n}{\partial p_n} =-p_n^2 + p_{n-1}^2 - 2p_{n-1} p_n + 4p_n - 2 = 0. \label{nth equation}
\end{equation}

If we put $p_n=\left[(n-1)\sqrt 2+1\right]p_1$ and $p_{n-1}=\left[(n-2)\sqrt 2+1\right]p_1$ into Equation (\ref{nth equation}), we obtain the following quadratic equation for $p_1$:
\begin{equation*}
-4p_{1}(p_{1}-1)(n-1) \sqrt{2} - 2 - 4(n -1)^{2}p_{1}^{2} + 4p_{1}=0.
\end{equation*}
By solving the quadratic equation, we find the unique optimal value for $p_1$ (the other value will lead to an infeasible solution), that is, 
$$ p_1=\frac{1}{(n-1)\sqrt 2+2}. $$
Therefore,  
$$p_i^*=\frac{(i-1)\sqrt 2+1}{(n-1)\sqrt 2+2}, \quad i \in \{1,\ldots,n\}.$$

\subsection{Objective Function Expression for Two Hubs in a Square}
\label{Appendix:OBJ Expression of Two hubs}

\begin{figure}[h!]
\centering
\includegraphics[width=0.25\textwidth]{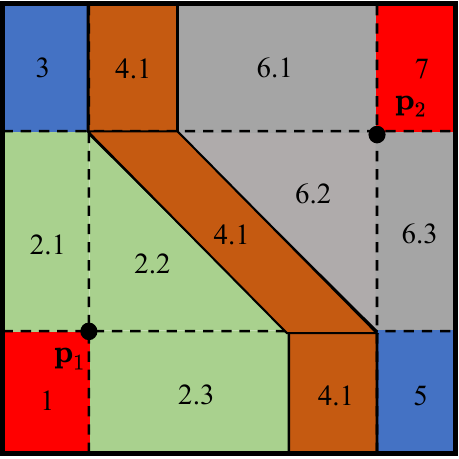}
\caption{Illustration of the Integration Area Division of $\mathbf{x}$ }
\label{CE}
\end{figure}

For the convenience of calculation, we further divide the potential area of $\mathbf{X}$ into 13 parts, as shown in Figure \ref{CE}. Moreover, as shown in Figure \ref{AP_all}, the potential area of $\mathbf{Y}$ is divided into the red area ($\mathbf{p}_1$ wins) and the blue area ($\mathbf{p}_2$ wins). $M$ and $N$ are the intersections of the separation line and the inner box. 

\begin{figure}[h!]
\centering
\includegraphics[width=0.6\textwidth]{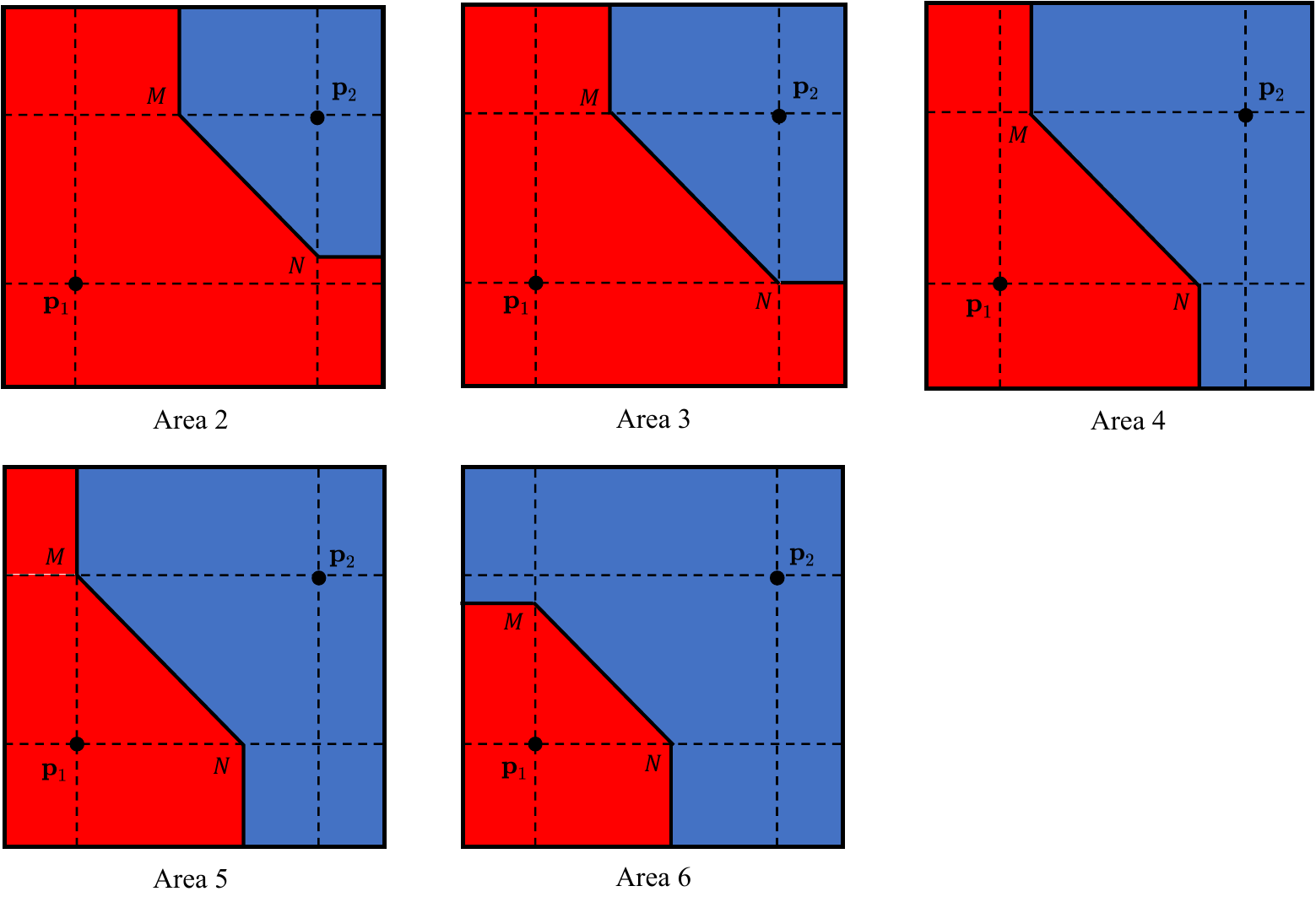}
\caption{Illustration of the Integration Area of $\mathbf{y}$ for $\mathbf{x}$ in Area 2 to Area 6}
\label{AP_all}
\end{figure}

\subsubsection{$(x_1,x_2)$ in Area 1}
According to Proposition \ref{2Dp1}, if $(x_1,x_2)$ is in Area 1, $\mathbf{p}_1$ always wins regardless of the location of $(y_1,y_2)$. Then, the corresponding integration expression can be obtained by:
\begin{equation}
\label{part1}
\int^{p_{12}}_0\int^{p_{11}}_0\int_0^1\int_0^1 d_1 \, dy_1 \, dy_2 \, dx_1 \,dx_2.
\end{equation}

\subsubsection{$(x_1,x_2)$ in Area 2}

\textbf{$(x_1,x_2)$ in Area 2.1}

When $(x_1,x_2)$ is in Area 2.1, the associated $b$ can be calculated as $b = -p_{11}+p_{12}+p_{21}+p_{22}-2x_2$. Therefore, the equation for the line segment $MN$ is $p_{12}+p_{21}+p_{22}-x_2=y_1+y_2$. Then, we obtain the locations of $M$ and $N$ as
$M=(p_{12}+p_{21}-x_2,p_{22})$ and $N=(p_{21},p_{12}+p_{22}-x_2)$. Let $A= p_{12}+p_{22}-x_2$, $ B = p_{12}+p_{21}+p_{22}-x_2-y_2$, and $ C = p_{12}+p_{21}-x_2$.

\begin{itemize}

\item Red area ($\mathbf{p}_1$ wins)
\begin{align}
\label{p1w}
&\int^{p_{22}}_{p_{12}}\int^{p_{11}}_0\int^{A}_{0}\int^{1}_0  d_1 \, dy_1 \, dy_2 \, dx_1 \, dx_2  \nonumber \\ 
&+\int^{p_{22}}_{p_{12}}\int^{p_{11}}_0\int^{p_{22}}_{A}\int^{B }_0  d_1 \, dy_1 \, dy_2 \, dx_1 \, dx_2  \nonumber \\   
&+\int^{p_{22}}_{p_{12}}\int^{p_{11}}_0\int^1_{p_{22}}\int^{C}_0 d_1 \, dy_1 \, dy_2 \, dx_1 \, dx_2  
\end{align}

\item Blue area ($\mathbf{p}_2$ wins)
\begin{align}
\label{p2w}
&\int^{p_{22}}_{p_{12}}\int^{p_{11}}_0\int^{p_{22}}_{A}\int^{1}_{B}  d_2 \, dy_1 \, dy_2 \, dx_1 \, dx_2  \nonumber \\ 
&+\int^{p_{22}}_{p_{12}}\int^{p_{11}}_0\int^1_{p_{22}}\int^{1}_{C} d_2 \, dy_1 \, dy_2 \, dx_1 \, dx_2  
\end{align}
\end{itemize}

\textbf{$(x_1,x_2)$ in Area 2.2}

The associated $b$ can be calculated as $b=p_{11}+p_{12}+p_{21}+p_{22}-2x_1-2x_2$. Therefore, the equation for the line segment $MN$ is 
$p_{11}+p_{12}+p_{21}+p_{22}-x_1-x_2=y_1+y_2$. 
Then, we obtain the locations of $M$ and $N$ as $M=(p_{11}+p_{12}+p_{21}-x_1-x_2,p_{22})$ and $N=(p_{21},p_{11}+p_{12}+p_{22}-x_1-x_2)$. Let $A= p_{11}+p_{12} +p_{22}-x_1-x_2$, $ B = p_{11}+p_{12}+p_{22}-x_1-x_2$, and $ C = p_{11}+p_{12}+p_{21}-x_1-x_2$.

\begin{itemize}

\item Red area ($\mathbf{p}_1$ wins)
\begin{align}
\label{p2.2.1}
&\int^{p_{22}}_{p_{12}}\int^{p_{11}+p_{22}-x_2}_{p_{11}}\int^{A}_{0}\int^{1}_0 d_1 \, dy_1 \, dy_2 \, dx_1 \, dx_2  \nonumber \\
&+\int^{p_{22}}_{p_{12}}\int^{p_{11}+p_{22}-x_2}_{p_{11}}\int^{p_{22}}_{A}\int^{B}_0 d_1 \, dy_1 \, dy_2 \, dx_1 \,dx_2   \nonumber \\ 
&+\int^{p_{22}}_{p_{12}}\int^{p_{11}+p_{22}-x_2}_{p_{11}}\int^1_{p_{22}}\int^{C}_0 d_1 \, dy_1 \, dy_2 \, dx_1 \, dx_2   
\end{align}

\item Blue area ($\mathbf{p}_2$ wins)
\begin{align}
\label{p2.2.2}
&\int^{p_{22}}_{p_{12}}\int^{p_{11}+p_{22}-x_2}_{p_{11}}\int^{p_{22}}_{A}\int^{1}_{B} d_2 \, dy_1 \, dy_2 \, dx_1 \, dx_2  \nonumber \\
&+\int^{p_{22}}_{p_{12}}\int^{p_{11}+p_{22}-x_2}_{p_{11}}\int^1_{p_{22}}\int^{1}_{C} d_2 \, dy_1 \, dy_2 \, dx_1 \, dx_2    
\end{align}
\end{itemize}

\textbf{$(x_1,x_2)$ in Area 2.3}

The associated $b$ can be calculated as $b=p_{11}-p_{12}+p_{21}+p_{22}-2x_1$. Therefore, the equation for the line segment $MN$ is $p_{11}+p_{21}+p_{22}-x_1=y_1+y_2$.
Then, we obtain the locations of $M$ and $N$ as $M=(p_{11}+p_{21}-x_1,p_{22})$ and $N=(p_{21},p_{11}+p_{22}-x_1)$. Let $A= p_{11}+p_{22}-x_1$, $ B = p_{11}+p_{21}+p_{22}-x_1-y_2$, and $ C = p_{11}+p_{21}-x_1$.

\begin{itemize}

\item Red area ($\mathbf{p}_1$ wins)
\begin{align}
\label{p2.3.1}
&\int^{p_{12}}_{0}\int^{p_{11}+p_{22}-p_{12}}_{p_{11}}\int^{A}_{0}\int^{1}_0 d_1 \, dy_1 \, dy_2 \, dx_1 \, dx_2 \nonumber \\ 
&+\int^{p_{12}}_{0}\int^{p_{11}+p_{22}-p_{12}}_{p_{11}}\int^{p_{22}}_{A}\int^{B}_0 d_1 \, dy_1 \, dy_2 \, dx_1 \, dx_2  \nonumber \\   
&+\int^{p_{12}}_{0}\int^{p_{11}+p_{22}-p_{12}}_{p_{11}}\int^1_{p_{22}}\int^{C}_0 d_1 \, dy_1 \, dy_2 \, dx_1 \, dx_2
\end{align}

\item Blue area ($\mathbf{p}_2$ wins)
\begin{align}
\label{p2.3.2}
&\int^{p_{12}}_{0}\int^{p_{11}+p_{22}-p_{12}}_{p_{11}}\int^{p_{22}}_{A}\int^{1}_{B} d_2 \, dy_1 \, dy_2 \, dx_1 \, dx_2 \nonumber \\
&+\int^{p_{12}}_{0}\int^{p_{11}+p_{22}-p_{12}}_{p_{11}}\int^1_{p_{22}}\int^{1}_{C} d_2 \, dy_1 \, dy_2 \, dx_1 \, dx_2
\end{align}
\end{itemize}

\subsubsection{$(x_1,x_2)$ in Area 3}

 The associated $b$ can be calculated as $b=-p_{11}+p_{12}+p_{21}-p_{22}$. Therefore, the equation for the line segment $MN$ is $p_{12}+p_{21}=y_1+y_2$. Then we obtain the locations of $M$ and $N$ as $M=(p_{12}+p_{21}-p_{22},p_{22})$ and $N=(p_{21},p_{12})$.

\begin{itemize}

\item Red area ($\mathbf{p}_1$ wins)
\begin{align}
\label{p3.1}
&\int^{1}_{p_{22}}\int^{p_{11}}_{0}\int^{p_{12}}_{0}\int^{1}_0 d_1 \, dy_1 \, dy_2 \, dx_1 \, dx_2 \nonumber \\
&+\int^{1}_{p_{22}}\int^{p_{11}}_{0}\int^{p_{22}}_{p_{12}}\int^{p_{12}+p_{21}-y_2}_0 d_1 \, dy_1 \, dy_2 \, dx_1 \, dx_2 \nonumber  \\  
&+\int^{1}_{p_{22}}\int^{p_{11}}_{0}\int^1_{p_{22}}\int^{p_{12}+p_{21}-p_{22}}_0 d_1 \, dy_1 \, dy_2 \, dx_1 \, dx_2 
\end{align}

\item Blue area ($\mathbf{p}_2$ wins)
\begin{align}
\label{p3.2}
&\int^{1}_{p_{22}}\int^{p_{11}}_{0}\int^{p_{22}}_{p_{12}}\int^{1}_{p_{12}+p_{21}-y_2} d_2 dy_1dy_2dx_1dx_2 \nonumber \\
&+\int^{1}_{p_{22}}\int^{p_{11}}_{0}\int^1_{p_{22}}\int^{1}_{p_{12}+p_{21}-p_{22}} d_2 dy_1dy_2dx_1dx_2 
\end{align}
\end{itemize}

\subsubsection{$(x_1,x_2)$ in Area 4}
\textbf{$(x_1,x_2)$ in Area 4.1}

The associated $b$ can be calculated as $b=p_{11}+p_{12}+p_{21}-p_{22}-2x_1$. Therefore, the equation for the line segment $MN$ is $p_{11}+p_{12}+p_{21}-x_1=y_1+y_2.$ Then, we obtain the locations of $M$ and $N$ as $M=(p_{11}+p_{12}+p_{21}-p_{22}-x_1,p_{22})$ and $N=(p_{11}+p_{21}-x_1,p_{12})$. Let $A= p_{11}+p_{21}-x_1$, $ B = p_{11}+p_{12}+p_{21}-x_1-y_2$, and $ C = p_{11}+p_{12}+p_{21}-p_{22}-x_1$.

\begin{itemize}

\item Red area ($\mathbf{p}_1$ wins)
\begin{align}
\label{p4.1.1}
&\int^{1}_{p_{22}}\int^{p_{12}+p_{21}-p_{22}}_{p_{11}}\int^{p_{12}}_{0}\int^{A}_0 d_1 \, dy_1 \, dy_2 \, dx_1 \, dx_2 \nonumber \\ 
&+\int^{1}_{p_{22}}\int^{p_{12}+p_{21}-p_{22}}_{p_{11}}\int^{p_{22}}_{p_{12}}\int^{B}_0 d_1 \, dy_1 \, dy_2 \, dx_1 \, dx_2 \nonumber \\   
&+\int^{1}_{p_{22}}\int^{p_{12}+p_{21}-p_{22}}_{p_{11}}\int^1_{p_{22}}\int^{C}_0 d_1 \, dy_1 \, dy_2 \, dx_1 \, dx_2 
\end{align}

\item Blue area ($\mathbf{p}_2$ wins)
\begin{align}
\label{p4.1.2}
&\int^{1}_{p_{22}}\int^{p_{12}+p_{21}-p_{22}}_{p_{11}}\int^{p_{12}}_{0}\int^{1}_{A} d_2 \, dy_1 \, dy_2 \, dx_1 \, dx_2 \nonumber \\ 
&+\int^{1}_{p_{22}}\int^{p_{12}+p_{21}-p_{22}}_{p_{11}}\int^{p_{22}}_{p_{12}}\int^{1}_{B} d_2 \, dy_1 \, dy_2 \, dx_1 \, dx_2 \nonumber  \\  
&+\int^{1}_{p_{22}}\int^{p_{12}+p_{21}-p_{22}}_{p_{11}}\int^1_{p_{22}}\int^{1}_{C} d_2 \, dy_1 \, dy_2 \, dx_1 \, dx_2 
\end{align}
\end{itemize}

\textbf{$(x_1,x_2)$ in Area 4.2}

The associated $b$ can be calculated as $b=p_{11}+p_{12}+p_{21}+p_{22}-2x_1-2x_2$. Therefore, the equation for the line segment $MN$ is $p_{11}+p_{12}+p_{21}+p_{22}-x_1-x_2=y_1+y_2$. 
Then, we obtain the locations of $M$ and $N$ as $M=(p_{11}+p_{12}+p_{21}-x_1-x_2,p_{22})$ and $N=(p_{11}+p_{21}+p_{22}-x_1-x_2,p_{12})$. Let $A= p_{11}+p_{21}+p_{22}-x_1-x_2$, $ B = p_{11}+p_{12}+p_{21}+p_{22}-x_1-x_2-y_2$, and $ C = p_{11}+p_{12}+p_{21}-x_1-x_2$.

\begin{itemize}

\item Red area ($\mathbf{p}_1$ wins)
\begin{align}
\label{p4.2.1}
&\int^{p_{22}}_{p_{12}}\int_{p_{11}+p_{22}-x_2}^{p_{12}+p_{21}-x_2}\int^{p_{12}}_{0}\int^{A}_0 d_1 \, dy_1 \, dy_2 \, dx_1 \, dx_2  \nonumber \\ 
&+\int^{p_{22}}_{p_{12}}\int_{p_{11}+p_{22}-x_2}^{p_{12}+p_{21}-x_2}\int^{p_{22}}_{p_{12}}\int^{B}_0 d_1 \, dy_1 \, dy_2 \, dx_1 \, dx_2  \nonumber \\ 
&+\int^{p_{22}}_{p_{12}}\int_{p_{11}+p_{22}-x_2}^{p_{12}+p_{21}-x_2}\int^1_{p_{22}}\int^{C}_0 d_1 \, dy_1 \, dy_2 \, dx_1 \, dx_2 
\end{align}

\item Blue area ($\mathbf{p}_2$ wins)
\begin{align}
\label{p4.2.2}
&\int^{p_{22}}_{p_{12}}\int_{p_{11}+p_{22}-x_2}^{p_{12}+p_{21}-x_2}\int^{p_{12}}_{0}\int^{1}_{A} d_2 \, dy_1 \, dy_2 \, dx_1 \, dx_2  \nonumber \\
&+\int^{p_{22}}_{p_{12}}\int_{p_{11}+p_{22}-x_2}^{p_{12}+p_{21}-x_2}\int^{p_{22}}_{p_{12}}\int^{1}_{B} d_2 \, dy_1 \, dy_2 \, dx_1 \,dx_2 \nonumber  \\ 
&+\int^{p_{22}}_{p_{12}}\int_{p_{11}+p_{22}-x_2}^{p_{12}+p_{21}-x_2}\int^1_{p_{22}}\int^{1}_{C} d_2 \, dy_1 \, dy_2 \, dx_1 \, dx_2 
\end{align}
\end{itemize}

\textbf{$(x_1,x_2)$ in Area 4.3}

The associated $b$ can be calculated as $b=p_{11}-p_{12}+p_{21}+p_{22}-2x_1$. Therefore, the equation for the line segment $MN$ is $p_{11}+p_{21}+p_{22}-x_1=y_1+y_2$.
Then, we obtain the locations of $M$ and $N$ as $M=(p_{11}+p_{21}-x_1,p_{22})$ and $N=(p_{11}+p_{21}+p_{22}-p_{12}-x_1,p_{12})$. Let $A= p_{11}+p_{21}+p_{22}-p_{12}-x_1$, $ B = p_{11}+p_{21}+p_{22}-x_1-y_2$, and $ C = p_{11}+p_{21}-x_1$.

\begin{itemize}

\item Red area ($\mathbf{p}_1$ wins)
\begin{align}
\label{p4.3.1}
&\int^{p_{12}}_{0}\int^{p_{21}}_{p_{11}+p_{22}-p_{12}}\int^{p_{12}}_{0}\int^{A}_0 d_1 \, dy_1 \, dy_2 \, dx_1 \, dx_2 \nonumber \\
&+\int^{p_{12}}_{0}\int^{p_{21}}_{p_{11}+p_{22}-p_{12}}\int^{p_{22}}_{p_{12}}\int^{B}_0 d_1 \, dy_1 \, dy_2 \, dx_1 \, dx_2 \nonumber \\ 
&+\int^{p_{12}}_{0}\int^{p_{21}}_{p_{11}+p_{22}-p_{12}}\int^1_{p_{22}}\int^{C}_0 d_1 \, dy_1 \, dy_2 \, dx_1 \, dx_2  
\end{align}

\item Blue area ($\mathbf{p}_2$ wins)
\begin{align}
\label{p4.3.2}
&\int^{p_{12}}_{0}\int^{p_{21}}_{p_{11}+p_{22}-p_{12}}\int^{p_{12}}_{0}\int^{1}_{A} d_2 \, dy_1 \, dy_2 \, dx_1 \,dx_2 \nonumber \\ 
&+\int^{p_{12}}_{0}\int^{p_{21}}_{p_{11}+p_{22}-p_{12}}\int^{p_{22}}_{p_{12}}\int^{1}_{B} d_2 \, dy_1 \, dy_2 \, dx_1 \, dx_2 \nonumber \\  
&+\int^{p_{12}}_{0}\int^{p_{21}}_{p_{11}+p_{22}-p_{12}}\int^1_{p_{22}}\int^{1}_{C} d_2 \, dy_1 \, dy_2 \, dx_1 \, dx_2 
\end{align}
\end{itemize}

\subsubsection{$(x_1,x_2)$ in Area 5}

The associated $b$ can be calculated as $b=p_{11}-p_{12}-p_{21}+p_{22}$. Therefore, the equation for the line segment $MN$ is
$p_{11}+p_{22}=y_1+y_2$. Then, we obtain the locations of $M$ and $N$ as $M=(p_{11},p_{22})$ and $N=(p_{11}+p_{22}-p_{12},p_{12})$.

\begin{itemize}

\item Red area ($\mathbf{p}_1$ wins)
\begin{align}
\label{p5.1}
&\int^{p_{12}}_{0}\int^{1}_{p_{21}}\int^{p_{12}}_{0}\int^{p_{11}+p_{22}-p_{12}}_0 d_1 \, dy_1 \, dy_2 \, dx_1 \, dx_2 \nonumber \\ 
&+\int^{p_{12}}_{0}\int^{1}_{p_{21}}\int^{p_{22}}_{p_{12}}\int^{p_{11}+p_{22}-y_2}_0 d_1 \, dy_1 \, dy_2 \, dx_1 \, dx_2  \nonumber \\ 
&+\int^{p_{12}}_{0}\int^{1}_{p_{21}}\int^1_{p_{22}}\int^{p_{11}}_0 d_1 \, dy_1 \, dy_2 \, dx_1 \, dx_2                 
\end{align}

\item Blue area ($\mathbf{p}_2$ wins)
\begin{align}
\label{p5.2}
&\int^{p_{12}}_{0}\int^{1}_{p_{21}}\int^{p_{12}}_{0}\int^{1}_{p_{11}+p_{22}-p_{12}} d_2 \, dy_1 \, dy_2 \, dx_1 \, dx_2 \nonumber \\ 
&+\int^{p_{12}}_{0}\int^{1}_{p_{21}}\int^{p_{22}}_{p_{12}}\int^{1}_{p_{11}+p_{22}-y_2} d_2 \, dy_1 \, dy_2 \, dx_1 \, dx_2 \nonumber \\         
&+\int^{p_{12}}_{0}\int^{1}_{p_{21}}\int^1_{p_{22}}\int^{1}_{p_{11}} d_2 \, dy_1 \, dy_2 \, dx_1 \, dx_2               
\end{align}
\end{itemize}

\subsubsection{$(x_1,x_2)$ in Area 6}

\textbf{$(x_1,x_2)$ in Area 6.1}

The associated $b$ can be calculated as $b=p_{11}+p_{12}+p_{21}-p_{22}-2x_1$. Therefore, the equation for the line segment $MN$ is
$p_{11}+p_{12}+p_{21}-x_1=y_1+y_2$.
Then, we obtain the locations of $M$ and $N$ as $M=(p_{11},p_{12}+p_{21}-x_1)$ and $N=(p_{11}+p_{21}-x_1,p_{12})$. Let $A= p_{11}+p_{21}-x_1$, $ B = p_{12}+p_{21}-x_1$, and $ C = p_{11}+p_{12}+p_{21}-x_1-y_2$.

\begin{itemize}

\item Red area ($\mathbf{p}_1$ wins)
\begin{align}
\label{p6.1.1}
&\int^{1}_{p_{22}}\int^{p_{21}}_{p_{12}+p_{21}-p_{22}}\int^{p_{12}}_{0}\int^{A}_0 d_1 \, dy_1 \, dy_2 \, dx_1 \,dx_2 \nonumber \\ 
&+\int^{1}_{p_{22}}\int^{p_{21}}_{p_{12}+p_{21}-p_{22}}\int^{B}_{p_{12}}\int^{C}_0 d_1 \, dy_1 \, dy_2 \, dx_1 \, dx_2 
\end{align}

\item Blue area ($\mathbf{p}_2$ wins)
\begin{align}
\label{p6.1.2}
&\int^{1}_{p_{22}}\int^{p_{21}}_{p_{12}+p_{21}-p_{22}}\int^{p_{12}}_{0}\int^{1}_{A} d_2 \, dy_1 \, dy_2 \, dx_1 \,dx_2 \nonumber \\
&+\int^{1}_{p_{22}}\int^{p_{21}}_{p_{12}+p_{21}-p_{22}}\int^{B}_{p_{12}}\int^{1}_{C} d_2 \, dy_1 \, dy_2 \,dx_1 \, dx_2 \nonumber  \\ 
&+\int^{1}_{p_{22}}\int^{p_{21}}_{p_{12}+p_{21}-p_{22}}\int^1_{B}\int^{1}_{0} d_2 \, dy_1 \, dy_2 \, dx_1 \, dx_2  
\end{align}
\end{itemize}

\textbf{$(x_1,x_2)$ in Area 6.2}

The associated $b$ can be calculated as $b=p_{11}+p_{12}+p_{21}+p_{22}-2x_1-2x_2$. Therefore, the equation for the line segment $MN$ is $p_{11}+p_{12}+p_{21}+p_{22}-x_1-x_2=y_1+y_2$.
Then, we obtain the locations of $M$ and $N$ as $M=(p_{11},p_{12}+p_{21}+p_{22}-x_1-x_2)$ and $N=(p_{11}+p_{21}+p_{22}-x_1-x_2,p_{12})$. Let $A= p_{11}+p_{21}+p_{22}-x_1-x_2$, $ B = p_{12}+p_{21}+p_{22}-x_1-x_2$, and $ C = p_{11}+p_{12}+p_{21}+p_{22}-x_1-x_2-y_2$.

\begin{itemize}

\item Red area ($\mathbf{p}_1$ wins)
\begin{align}
\label{p6.2.1}
&\int^{p_{22}}_{p_{12}}\int^{p_{21}}_{p_{12}+p_{21}-x_2}\int^{p_{12}}_{0}\int^{A}_0 d_1 \, dy_1 \, dy_2 \, dx_1 \, dx_2 \nonumber \\
&+\int^{p_{22}}_{p_{12}}\int^{p_{21}}_{p_{12}+p_{21}-x_2}\int^{B}_{p_{12}}\int^{C}_0 d_1 \, dy_1 \, dy_2 \, dx_1 \, dx_2 
\end{align}

\item Blue area ($\mathbf{p}_2$ wins)
\begin{align}
\label{p6.2.2}
&\int^{p_{22}}_{p_{12}}\int^{p_{21}}_{p_{12}+p_{21}-x_2}\int^{p_{12}}_{0}\int^{1}_{A} d_2 \, dy_1 \, dy_2 \, dx_1 \, dx_2 \nonumber \\ 
&+\int^{p_{22}}_{p_{12}}\int^{p_{21}}_{p_{12}+p_{21}-x_2}\int^{B}_{p_{12}}\int^{1}_{C} d_2 \, dy_1 \, dy_2 \, dx_1 \, dx_2  \nonumber \\  
&+\int^{p_{22}}_{p_{12}}\int^{p_{21}}_{p_{12}+p_{21}-x_2}\int^1_{B}\int^{1}_{0} d_2 \, dy_1 \, dy_2 \, dx_1 \, dx_2 
\end{align}
\end{itemize}

\textbf{$(x_1,x_2)$ in Area 6.3}

The associated $b$ can be calculated as $b=p_{11}+p_{12}-p_{21}+p_{22}-2x_2$. Therefore, the equation for the line segment $MN$ is $p_{11}+p_{12}+p_{22}-x_2=y_1+y_2$.
Then, we obtain the locations of $M$ and $N$ as $M=(p_{11},p_{12}+p_{22}-x_2)$ and $N=(p_{11}+p_{22}-x_2,p_{12})$. Let $A= p_{11}+p_{22}-x_2$, $ B = p_{12}+p_{22}-x_2$, and $ C = p_{11}+p_{12}+p_{22}-x_2-y_2$.

\begin{itemize}

\item Red area ($\mathbf{p}_1$ wins)
\begin{align}
\label{p6.3.1}
&\int^{p_{22}}_{p_{12}}\int^{1}_{p_{21}}\int^{p_{12}}_{0}\int^{A}_0 d_1 \, dy_1 \, dy_2 \, dx_1 \, dx_2 \nonumber \\ 
&+\int^{p_{22}}_{p_{12}}\int^{1}_{p_{21}}\int^{B}_{p_{12}}\int^{C}_0 d_1 \, dy_1 \, dy_2 \, dx_1 \, dx_2        
\end{align}

\item Blue area ($\mathbf{p}_2$ wins)
\begin{align}
\label{p6.3.2}
&\int^{p_{22}}_{p_{12}}\int^{1}_{p_{21}}\int^{p_{12}}_{0}\int^{1}_{A} d_2 \, dy_1 \, dy_2 \, dx_1 \, dx_2 \nonumber \\
&+\int^{p_{22}}_{p_{12}}\int^{1}_{p_{21}}\int^{B}_{p_{12}}\int^{1}_{C} d_2 \, dy_1 \, dy_2 \, dx_1 \, dx_2 \nonumber \\        
&+\int^{p_{22}}_{p_{12}}\int^{1}_{p_{21}}\int^1_{B}\int^{1}_{0} d_2 \, dy_1 \, dy_2 \, dx_1 \, dx_2         
\end{align}
\end{itemize}

\subsubsection{$(x_1,x_2)$ in Area 7}

According to Proposition \ref{2Dp2}, when $X$ is in Part 7, $d_2 = \min \{d_1, d_2\}$wherever $Y$ is. Then, the expression of Part 7 can be obtained by:
\begin{equation}
\label{part7}
\int^{1}_{p_{22}}\int^{1}_{p_{21}}\int_{0}^1\int_0^1 d_2 \, dy_1 \, dy_2 \, dx_1 \,dx_2.
\end{equation}

Therefore, the final expression can be obtained by summing up from Equation (\ref{part1}) to Equation (\ref{part7}).

\subsection{ Analytic Expression of $F_2$}
\label{Appendix:F_2}

\begin{align*}
F_2=&2-2p_{22}-2p_{21}-\frac{2}{3} p_{12}^{3} p_{11} p_{22}+\frac{1}{3} p_{11} p_{12}^{4}\\&-\frac{1}{3} p_{22}^{4} p_{11}-p_{11}^{2} p_{22}+\frac{2}{3} p_{12}^{3} p_{21}-p_{11}^{2} p_{12}^{2}\\
&+p_{11}^{2} p_{22}^{2}+\frac{1}{3} p_{12}^{4} p_{22}-\frac{1}{3} p_{21} p_{22}^{4}-\frac{1}{3} p_{12} p_{22}^{4}\\&+p_{11}^{2} p_{21}-p_{11} p_{21}^{2}+2 p_{12} p_{11} p_{22}\\
&+\frac{2}{3} p_{22}^{3} p_{11} p_{12}-p_{21} p_{22}^{2}+p_{21}^{2} p_{22}\\&+\frac{2}{3} p_{11} p_{22}^{3}-p_{21}^{2} p_{22}^{2}+\frac{4}{3} p_{21} p_{22}^{3}-p_{11} p_{22}^{2}\\
&+p_{12}^{2} p_{21}^{2}+p_{12}^{2} p_{22}-\frac{1}{3} p_{21}^{3}+2 p_{22}^{2}-2 p_{11} p_{12}^{2} p_{21} p_{22}\\&+\frac{2}{3} p_{22}^{3} p_{12}+\frac{2}{3} p_{22}^{3} p_{12} p_{21}-p_{21} p_{12}^{2}\\
&+\frac{4}{3} p_{12}^{3} p_{11}-\frac{2}{3} p_{12}^{3} p_{11} p_{21}-\frac{2}{3} p_{12}^{3} p_{21} p_{22}\\&-2 p_{12} p_{21} p_{22}^{2}-4 p_{11} p_{12} p_{22}^{2}\\
&+\frac{1}{3} p_{12}^{4} p_{21}+\frac{2}{3} p_{12}^{3} p_{22}-p_{12} p_{22}^{2}-\frac{2}{3} p_{12}^{4}-p_{21}^{2} p_{12}\\&+2 p_{21}^{2}+\frac{1}{3} p_{22}^{3}-\frac{1}{3} p_{12}^{3}-p_{12}^{2} p_{11}+p_{11}^{2} p_{12}\\
&+2 p_{11} p_{12} p_{21} p_{22}^{2}-\frac{2}{3} p_{22}^{4}+\frac{1}{5} p_{22}^{5}\\&-\frac{1}{5} p_{12}^{5}+\frac{1}{3} p_{11}^{3}-2 p_{12} p_{11} p_{21}+\frac{2}{3} p_{11} p_{21} p_{22}^{3}\\
&-2 p_{11} p_{21} p_{22}^{2}+2 p_{11} p_{21} p_{22}+2 p_{12} p_{21} p_{22}\\&+2 p_{11} p_{12}^{2} p_{21}+2 p_{11} p_{12}^{2} p_{22}.
\end{align*}

\subsection{Two Hubs in a Non-Square Rectangle}
\label{Appendix:two hub rectangle}

Here, we focus on a non-square rectangle with width one and positive length $a \neq 1$. Similar to the square situation, we start by assuming that $\mathbf{p}_2$ is on the northeast side of $\mathbf{p}_1$. As illustrated in Figure \ref{2p8}, there are two cases to be considered. In the first case, $\Delta_1 \geq \Delta_2$, while in the second case, $\Delta_1 < \Delta_2$. Since we do not know which case will result in the optimal two-hub location solution, we use a similar derivation procedure used in the square situation to obtain the analytic expressions for the objective functions of both cases. Moreover, through comparing the optimal solutions of the two cases, which are obtained by solving the corresponding nonlinear programs with BARON, we obtain the optimal solution for our non-square rectangle situation.     

\begin{figure}[htbp]
    \centering
    \includegraphics[width=0.6\textwidth]{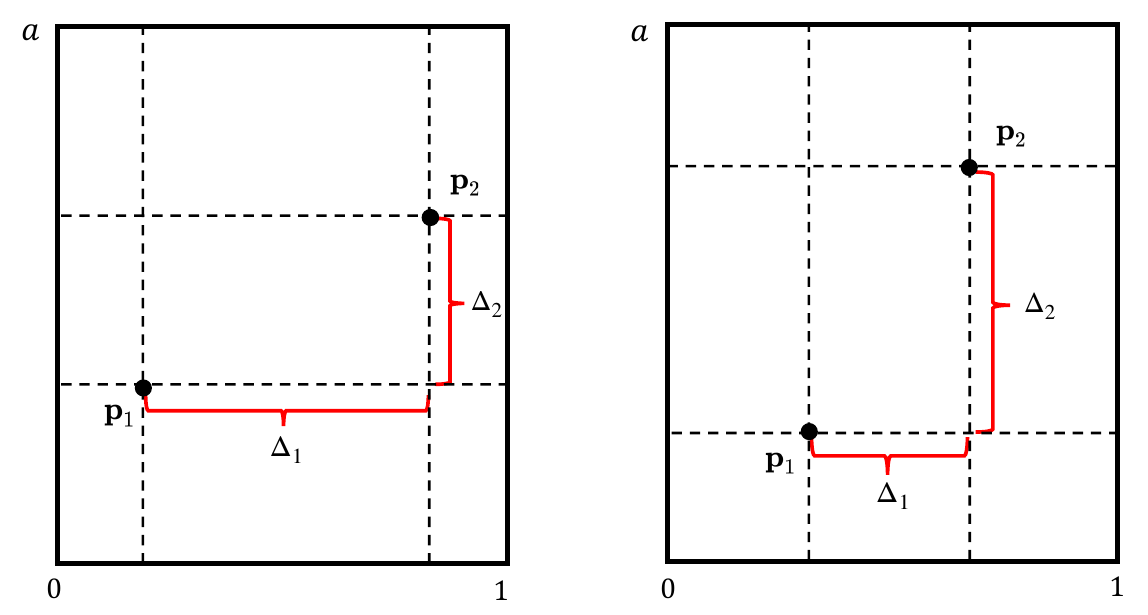}
    \caption{Case 1 (left): $\Delta_1 \ge \Delta_2$ and Case 2 (right): $\Delta_1 < \Delta_2$}
    \label{2p8}
\end{figure}

The objective function expression for Case 1 is
\begin{align*}
    F_2^{C1} = & \frac{1}{a^2}\bigg(\frac{2}{3} a p_{11} p_{12}^{3}+\frac{4}{3} a p_{11} p_{22}^{3}+\frac{2}{3} p_{11} p_{12} p_{22}^{3}\\
    &+\frac{2}{3} p_{11} p_{21} p_{22}^{3}-\frac{2}{3} p_{11} p_{12}^{3} p_{21}-2 p_{11} p_{12} p_{22}^{2}\\
    &+2 p_{11} p_{12}^{2} p_{22}+\frac{2}{3} p_{12}^{3} p_{11}-\frac{1}{3} p_{11} p_{22}^{4}+\frac{1}{3} p_{11} p_{12}^{4}\\
    &-\frac{2}{3} p_{22}^{3} p_{11}+\frac{2}{3} p_{12}^{3} p_{22}+p_{12}^{2} p_{22}^{2}+2 a p_{22}^{2}\\
    &-\frac{2}{3} p_{12} p_{22}^{3}+\frac{1}{3} p_{21} p_{12}^{4}+\frac{1}{3} p_{12}^{4} p_{22}-\frac{1}{3} p_{21} p_{22}^{4}\\
    &+\frac{1}{6} a p_{12}^{4}-\frac{1}{2} a p_{22}^{4}-\frac{1}{3} p_{12} p_{22}^{4}+\frac{1}{3} a^{2} p_{11}^{3}+2 a^{2} p_{21}^{2}\\
    &-\frac{1}{3} a^{2} p_{21}^{3}-2 a^{2} p_{21}-2 a^{2} p_{22}-\frac{1}{3} a^{2} p_{12}^{3}+\frac{1}{3} a^{2} p_{22}^{3}\\
    &-\frac{5}{6} p_{12}^{4}+\frac{1}{5} p_{22}^{5}-\frac{1}{5} p_{12}^{5}-\frac{1}{6} p_{22}^{4}+a^{3}+a^{2}\\
    &+2 p_{11} p_{12} p_{21} p_{22}^{2}-2 a p_{11} p_{12} p_{22}^{2}+2 a p_{11} p_{12}^{2} p_{21}\\
    &-2 a p_{11} p_{21} p_{22}^{2}-2 a p_{12} p_{21} p_{22}^{2}-2 p_{11} p_{12}^{2} p_{21} p_{22}\\
    &+2 a^{2} p_{11} p_{12} p_{22}+2 a^{2} p_{11} p_{21} p_{22}-2 a^{2} p_{11} p_{12} p_{21}\\
    &+2 a^{2} p_{12} p_{21} p_{22}-a p_{12}^{2} p_{22}^{2}-a p_{11}^{2} p_{12}^{2}+a p_{11}^{2} p_{22}^{2}\\
    &+\frac{2}{3} a p_{12}^{3} p_{21}+\frac{4}{3} p_{21} a p_{22}^{3}+\frac{2}{3} p_{21} p_{12} p_{22}^{3}+\frac{4}{3} p_{12} a p_{22}^{3}\\
    &-\frac{2}{3} p_{21} p_{22} p_{12}^{3}-\frac{2}{3} p_{12}^{3} p_{11} p_{22}-a^{2} p_{11} p_{22}^{2}\\
    &-a^{2} p_{11} p_{12}^{2}-a^{2} p_{11} p_{21}^{2}-a p_{21}^{2} p_{22}^{2}-a^{2} p_{12}^{2} p_{21}\\
    &+a^{2} p_{12}^{2} p_{22}-a^{2} p_{12} p_{22}^{2}-a^{2} p_{21} p_{22}^{2}\\
    &+a^{2} p_{11}^{2} p_{12}+a^{2} p_{11}^{2} p_{21}-a^{2} p_{11}^{2} p_{22}\\
    &+a p_{12}^{2} p_{21}^{2}-a^{2} p_{12} p_{21}^{2}+a^{2} p_{21}^{2} p_{22}\bigg).
\end{align*}

The objective function expression for Case 2 is
\begin{align*}
    F_2^{C2} = & \frac{1}{a^2}\bigg(\frac{2}{3} p_{11}^{3} p_{12} a +\frac{2}{3} p_{11} p_{21}^{3} p_{22}-2 p_{11} p_{21}^{2} p_{22}\\
    &+\frac{2}{3} a p_{11}^{3} p_{21}+a p_{11}^{2} p_{21}^{2}-\frac{2}{3} a p_{11} p_{21}^{3}+\frac{1}{3} p_{11}^{4} p_{22}\\
    &-\frac{1}{3} p_{11} p_{21}^{4}-p_{11}^{2} p_{21}^{2}+\frac{4}{3} p_{11} p_{21}^{3}-\frac{5}{6} a p_{11}^{4}\\
    &+\frac{2}{3} p_{11}^{3} p_{22}-p_{11}^{2} p_{12}^{2}+\frac{2}{3} p_{11}^{3} p_{12}+\frac{4}{3} p_{21}^{3} p_{22}+2 a^{2} p_{21}^{2}\\
    &-\frac{1}{3} p_{21}^{4} p_{22}+p_{21}^{2} p_{12}^{2}-\frac{1}{6} a p_{21}^{4}+\frac{1}{3} p_{11}^{4} p_{21}+\frac{4}{3} p_{12} p_{21}^{3}\\
    &-\frac{1}{3} p_{21}^{4} p_{12}+\frac{1}{3} p_{11}^{4} p_{12}-p_{21}^{2} p_{22}^{2}-p_{11}^{2} p_{22}+p_{12}^{2} p_{22}\\
    &-p_{12} p_{22}^{2}+p_{11}^{2} p_{22}^{2}-2 a^{2} p_{21}-2 a^{2} p_{22}+2 a p_{22}^{2}\\
    &-p_{12}^{2} p_{21}+p_{11}^{2} p_{21}-p_{11} p_{21}^{2}-p_{21}^{2} p_{22}+p_{11} p_{12}^{2}\\
    &-p_{21}^{2} p_{12}-p_{11}^{2} p_{12}+p_{21} p_{22}^{2}-p_{11} p_{22}^{2}+\frac{1}{6} p_{11}^{4}\\
    &-\frac{1}{5} p_{11}^{5}-\frac{1}{2} p_{21}^{4}+\frac{1}{5} p_{21}^{5}-\frac{1}{3} p_{11}^{3}-\frac{1}{3} p_{22}^{3}+\frac{1}{3} p_{21}^{3}\\
    &+\frac{1}{3} p_{12}^{3}+a^{3}+a^{2}+2 a p_{11}^{2} p_{12} p_{21}-2 a p_{11} p_{12} p_{21}^{2}\\
    &+2 p_{11} p_{12} p_{21}^{2} p_{22}-2 p_{11}^{2} p_{12} p_{21} p_{22}-\frac{2}{3} p_{21} p_{22} p_{11}^{3}\\
    &+\frac{2}{3} p_{21}^{3} p_{11} p_{12}-\frac{2}{3} a p_{12} p_{21}^{3}-\frac{2}{3} p_{11}^{3} p_{12} p_{22} \\
    &+\frac{2}{3} p_{12} p_{21}^{3} p_{22}+2 p_{11}^{2} p_{12} p_{22}-2 p_{12} p_{21}^{2} p_{22}\\
    &-2 p_{21}^{2} p_{11} p_{12}-\frac{2}{3} p_{11}^{3} p_{12} p_{21}+2 p_{22} p_{11} p_{21}\\
    &+2 p_{12} p_{21} p_{22} +2 p_{12} p_{11} p_{21}-2 p_{11} p_{22} p_{12}\bigg).
\end{align*}

For different values of $a$, Table \ref{T1} shows the associated optimal objective function values for both Case 1 and Case 2. When comparing these two values, we observe that when $a= 1/5$, $1/4$, $1/3$, and $1/2$, Case 1 (that is, $\Delta_1 \geq \Delta_2$) will result in a smaller objective function value, while when $a= 2$, 3, 4, and 5, Case 2 (that is, $\Delta_1 < \Delta_2$) will result in a smaller value. We conjecture that Case 1 will result in an optimal solution when $a<1$ and Case 2 will result in an optimal solution when $a>1$.

\begin{table}[h!]
\centering
\caption{Comparison Results for Different $a$ Values}
\label{T1}
\begin{tabular}{@{}lcccccccc@{}}
\toprule
$a$      & 1/5       & 1/4       & 1/3       & 1/2       & 2         & 3         & 4         & 5         \\ 
\midrule
Case 1   & \bf{0.4896} & \bf{0.5143} & \bf{0.5554} & \bf{0.6371} & 1.3338    & 1.8125    & 2.2996    & 2.7909    \\
Case 2   & 0.5582    & 0.5749    & 0.6041    & 0.6669    & \bf{1.2742} & \bf{1.6662} & \bf{2.0573} & \bf{2.4481} \\ 
\bottomrule
\end{tabular}
\end{table}

For any given $a$, if $\mathbf{p}_2$ is in the northwest side of $\mathbf{p}_1$, we obtain another optimal solution pair. For different $a$ values, optimal two-hub locations are summarized in Table \ref{T2}. We observe that for each optimal solution pair, $\mathbf{p}_1^*$ and $\mathbf{p}_2^*$ are symmetric with respect to $(1/2, a/2)$, which is the center of the rectangle.

\begin{table}[h!]
\centering
\caption{Optimal Hub Locations for Different $a$ Values}
\label{T2}
\begin{tabular}{@{}lcccccccccc@{}}
\toprule
$a$  & $p_{11}^*$ & $p_{12}^*$ & $p_{21}^*$ & $p_{22}^*$ & $a$  & $p_{11}^*$ & $p_{12}^*$ & $p_{21}^*$ & $p_{22}^*$ \\ 
\midrule
\multirow{2}{*}{$1/5$} & 0.2951     & 0.0895     & 0.7049     & 0.1105     & \multirow{2}{*}{2} & 0.4250     & 0.6017     & 0.5750     & 1.3983     \\
                       & 0.2951     & 0.1105     & 0.7049     & 0.0895     &                  & 0.4250     & 1.3983     & 0.5750     & 0.6017     \\
\multirow{2}{*}{$1/4$} & 0.2958     & 0.1112     & 0.7042     & 0.1388     & \multirow{2}{*}{3} & 0.4390     & 0.8915     & 0.5610     & 2.1085     \\
                       & 0.2958     & 0.1388     & 0.7042     & 0.1112     &                  & 0.4390     & 2.1085     & 0.5610     & 0.8915     \\
\multirow{2}{*}{$1/3$} & 0.2972     & 0.1463     & 0.7028     & 0.1870     & \multirow{2}{*}{4} & 0.4447     & 1.1831     & 0.5553     & 2.8169     \\
                       & 0.2972     & 0.1870     & 0.7028     & 0.1463     &                  & 0.4447     & 2.8169     & 0.5553     & 1.1831     \\
\multirow{2}{*}{$1/2$} & 0.3009     & 0.2125     & 0.6991     & 0.2875     & \multirow{2}{*}{5} & 0.4477     & 1.4754     & 0.5523     & 3.5246     \\
                       & 0.3009     & 0.2875     & 0.6991     & 0.2125     &                  & 0.4477     & 3.5246     & 0.5523     & 1.4754     \\ 
\bottomrule
\end{tabular}
\end{table}



\end{document}